\documentclass[11pt,reqno,twoside]{amsart}
\usepackage{graphicx}
\usepackage{amssymb}
\usepackage{epstopdf}
\usepackage{cite}
\usepackage[asymmetric,top=3.5cm,bottom=4.3cm,left=3.1cm,right=3.1cm]{geometry}
\usepackage{bm}

\usepackage{booktabs} 
\usepackage{array} 
\usepackage{paralist} 
\usepackage{verbatim} 
\usepackage{subfigure} 
\usepackage{tabularx}
\usepackage{amsmath,amsfonts,amsthm,mathrsfs,amssymb,cite}
\usepackage[usenames]{color}
\usepackage{ulem}

\newtheorem{thm}{Theorem}[section]

\newtheorem{lem}{Lemma}[section]
\newtheorem{prop}{Proposition}[section]
\theoremstyle{definition}
\newtheorem{defn}{Definition}[section]
\theoremstyle{remark}

\newtheorem{rem}{Remark}[section]
\numberwithin{equation}{section}

\newcommand{\supp}{{\rm supp}}

\newcommand{\rmd}{\mathrm {d}}

\newcommand{\bu}{\mathbf{u}}

\newcommand{\bv}{\mathbf{v}}
\newcommand{\bw}{\mathbf{w}}
\newcommand{\bff}{\mathbf{f}}

\newcommand{\bx}{\mathbf{x}}
\newcommand{\by}{\mathbf{y}}

\newcommand{\bd}{\mathbf{d}}
\newcommand{\bp}{\mathbf{p}}
\newcommand{\bH}{\mathbf{H}}

\newcommand{\bQ}{\mathbf{Q}}

\newcommand{\bg}{\mathbf{g}}
\newcommand{\bh}{\mathbf{h}}
\newcommand{\bq}{\mathbf{q}}
\newcommand{\bW}{\mathbf{W}}
\newcommand{\bV}{\mathbf{V}}
\newcommand{\bU}{\mathbf{U}}
\newcommand{\ba}{\mathbf{a}}
\newcommand{\bb}{\mathbf{b}}
\newcommand{\bc}{\mathbf{c}}

\newcommand{\bpsi}{\bm{\psi}}

\newcommand{\bxi}{\bm{\xi}}
\newcommand{\bnu}{\bm{\nu}}
\newcommand{\blambda}{\bm{\lambda}}
\newcommand{\bPhi}{\bm{\Phi}}

\newcommand{\bE}{\mathbf{E}}

\newcommand{\bJ}{\mathbf{J}}

\newcommand{\beq}{\begin{equation}}
\newcommand{\eeq}{\end{equation}}
\allowdisplaybreaks

\title[Effective medium theory for electromagnetic scattering]
{Effective medium theory for embedded obstacles in electromagnetic scattering with applications}

\author{Huaian Diao}
\address{School of Mathematics, Jilin University, Changchun, Jilin, China.}
\email{diao@jlu.edu.cn, hadiao@gmail.com}
\author{Hongyu Liu}
\address{Department of Mathematics, City University of Hong Kong, Kowloon Tong, Hong Kong SAR, China.}
\email{hongyu.liuip@gmail.com; hongyliu@cityu.edu.hk}
\author{Qingle Meng}
\address{Department of Mathematics, City University of Hong Kong, Kowloon Tong, Hong Kong SAR, China.}
\email{mengql2021@foxmail.com; qinmeng@cityu.edu.hk}

\author{Li Wang}
\address{Department of Mathematics, City University of Hong Kong, Kowloon, Hong Kong SAR, China.}
\email{lwang637-c@my.cityu.edu.hk}

\begin{document}
\maketitle
\begin{abstract}

This paper focuses on the time-harmonic electromagnetic (EM) scattering problem in a general medium which may possess a nontrivial topological structure. We model this by an inhomogeneous and possibly anisotropic medium with embedded obstacles and the EM waves cannot penetrate inside the obstacles. Such a situation naturally arises in studying inverse EM scattering problems from complex mediums with partial boundary measurements, or inverse problems from EM mediums with metal inclusions. We develop a novel theoretical framework by showing that the embedded obstacles can be effectively approximated by a certain isotropic medium with a specific choice of material parameters. We derive sharp estimates to verify this effective approximation and also discuss the practical implications of our results to the inverse problems mentioned above, which are longstanding topics in inverse scattering theory.

\medskip

\medskip



\end{abstract}

\section{Introduction}

\subsection{Mathematical setup}\label{sub:MS}

Initially, our focus lies on establishing the mathematical framework for our subsequent study. Specifically, we are interested in the time-harmonic electromagnetic (EM) scattering problem occurring within a general medium, which might exhibit a nontrivial topological structure; see Fig.~\ref{fig:1} for a schematic illustration.
\begin{figure}[t]
\centering
    {\includegraphics[width=0.4\textwidth]{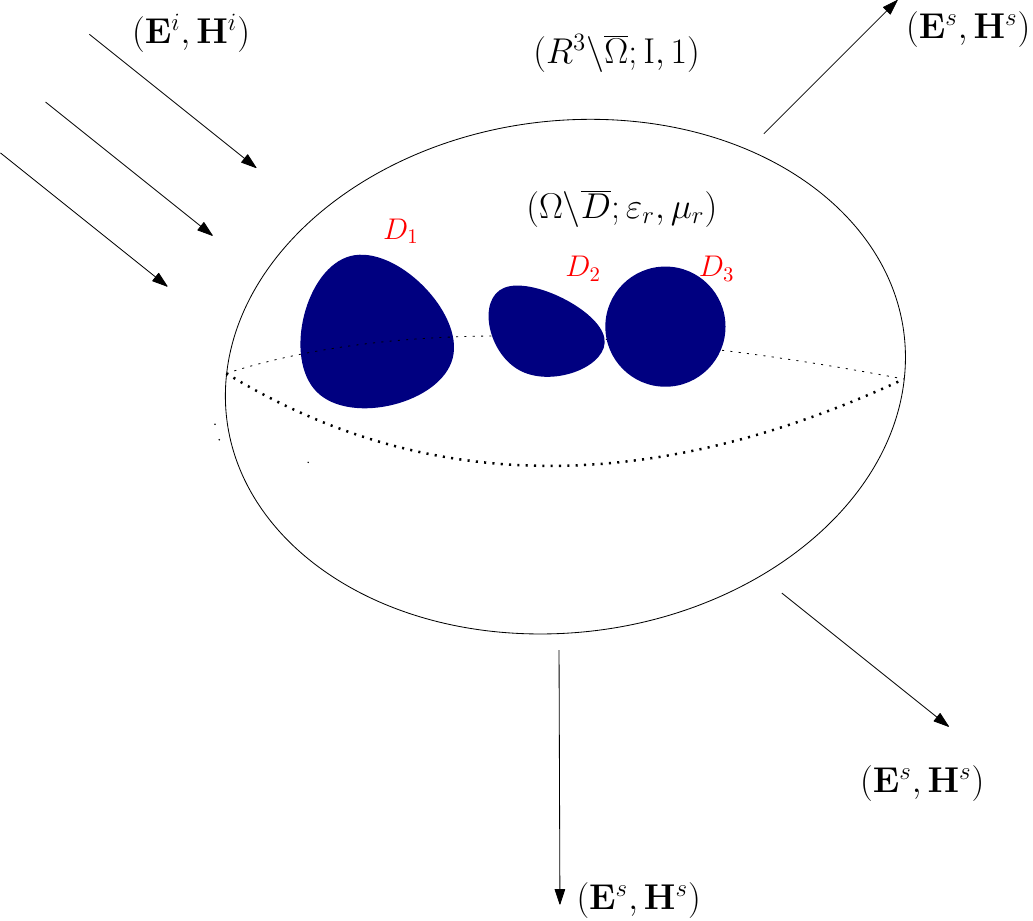}}
    \caption{Schematic illustration of the anisotropic electromagnetic scattering problem with multiple embedded obstacles. \label{fig:1}  }
\end{figure}
Let $\Omega$ be bounded Lipschitz domains such that $\mathbb{R}^3\backslash\overline{\Omega}$  is connected. Suppose that $D\Subset \Omega$ represents the impenetrable obstacle such that $\Omega\backslash\overline{D}$ is connected, which could consist of a finite number of pairwise disjoint obstacles $D_l$ ($l=1,\ldots,N$, $N\in \mathbb N$). Namely, $D=\cup_{l=1}^N D_l$ and $D_i\cap D_j=\emptyset $ for $i\neq j$ and $i,j\in \{1,\ldots,N\}$. The medium configuration is characterised by the relative permittivity $\varepsilon_r(\bx)$ and the relative permeability $\mu_r(\bx)$, which are defined as follows:
$$
\varepsilon_r=\frac{\varepsilon+\mathrm{i}\sigma/\omega}{\varepsilon_0}\quad\mbox{and}\quad \mu_r=\frac{\mu}{\mu_0}\quad\mbox{for}\quad \bx\in\mathbb{R}^3\backslash\overline{D}.
$$
In the above expressions, $\omega$ denotes the temporal frequency of the electromagnetic waves,  $\varepsilon_0$ and $\mu_0$ are positive constants representing the electric permittivity and magnetic permeability of the medium in $\mathbb{R}^3\backslash\overline{\Omega}$, respectively. Furthermore, $\varepsilon$, $\mu$, and $\sigma$ correspond to the electric permittivity, magnetic permeability, and electric conductivity of the possibly anisotropic medium in $\Omega\backslash\overline{D}$, respectively. Here, $\varepsilon$ and $\sigma$ are $3 \times 3$ matrix-valued functions, and $\mu$ is a scalar function. We assume that $\mu_r(\bx)$ belongs to $L^\infty(\Omega\backslash\overline{D})$ and $\varepsilon_r\in L^\infty(\Omega\backslash\overline{D})^{3\times 3}$ is symmetric, and satisfy
\begin{align}\label{ineq:gamma}
0<\gamma_1\leq \mu_r\leq\gamma_2,\quad \overline{\bxi}\cdot\Re \varepsilon_r\,\bxi\geq\gamma|\bxi|^2>0\quad\mbox{and}\quad {\bxi}\cdot\Im \varepsilon_r\,\bxi\geq 0
\end{align}
for all $\bx \in \Omega\backslash\overline{D}$ and all $\bxi\in \mathbb{C}^3$, where $\gamma_1$,$\gamma_2$ and $\gamma$ are positive constants.

 In this paper, we take the time-harmonic electromagnetic incident fields of the form
\begin{align*}
\bE^i=\bp \, e^{\mathrm{i}k\bx\cdot\bd}\quad\mbox{and}\quad
 \bH^i=(\bd\times\bp)\,e^{\mathrm{i}k\bx\cdot\bd}\quad \mbox{for}\quad \bx\in\mathbb{R}^3,
 \end{align*}
 where $\bd\in\mathbb{S}^2$ is the propagation direction, $\mathrm{i}:=\sqrt{-1}$, $\bp\in\mathbb{S}^2$ is the polarization vector satisfying $\bp\cdot\bd=0$, and the wave number $k:=\sqrt{\mu_0\varepsilon_0}\,\omega$ is positive.
 Due to the presence of the complex scatterer, denoted by $D\oplus(\Omega\backslash
\overline{D}; \mu_r, \varepsilon_r)$, the electric scattered field denoted by $\bE^s$ and the magnetic scattered field denoted by $\bH^s$ characterize the perturbation of the propagation of the incident fields. We next consider the following EM scattering problem:
\begin{equation}\label{eq:sca1}
\begin{cases}
\nabla\times \bE-\mathrm{i}k\mu_r\bH={\bf 0},\quad \nabla\times \bH+\mathrm{i}k\varepsilon_r\bE={\bf0}&\mbox{in}\ \ \Omega\backslash\overline{D},\medskip\\
\nabla\times \bE^s-\mathrm{i}k\bH^s={\bf 0},\,\quad \nabla\times \bH^s+\mathrm{i}k\bE^s={\bJ}& \mbox{in}\ \ \mathbb{R}^3\backslash\overline{\Omega},\medskip\\\
\bnu\times \bE\big|_{\partial \Omega}=\bnu\times \bE^s\big|_{\partial \Omega}+\bnu\times \bE^{i}\big|_{\partial \Omega}&\mbox{on}\ \ \partial{\Omega},\\
\bnu\times \bH\big|_{\partial \Omega}=\bnu\times \bH^s\big|_{\partial \Omega}+\bnu\times \bH^{i}\big|_{\partial \Omega}&\mbox{on}\ \ \partial{\Omega},\\
\mathfrak{B}(\bE,\bH)={\bf0} & \mbox{on}\ \ \partial D,\medskip\\
\lim\limits_{r \to\infty}r\big(\bE^s \times\hat{\bx}+\bH^s\big)={\bf 0},
\end{cases}
\end{equation}
where $\bJ$ is the given applied current density, and $\mathfrak{B}(\mathbf{E},\mathbf{H}) =:\bnu \times \mathbf{E}$ or $\mathfrak{B}(\mathbf{E},\mathbf{H}) =: \bnu\times \mathbf{H}$, corresponding to the perfectly electrically conducting (PEC) or perfectly magnetically conducting (PMC) obstacle.
The last limit in (\ref{eq:sca1}) is known as the Silver-M\"{u}ller radiation condition, which describes that the scattered fields are outgoing. Based on the radiation condition mentioned above, the electric scattered fields $\bE^s\in H_{\mathrm{loc}}^1(\mathrm{curl}, \mathbb{R}^3\backslash\overline{\Omega})$ and $\bH^s\in H_{\mathrm{loc}}^1(\mathrm{curl},\mathbb{R}^3\backslash\overline{\Omega})$ admit the following asymptotic behaviour (cf. \cite{Colton2013,Monk2002}):
\begin{align*}
\bQ^s(\bx,\bd,\bp)=\frac{e^{\mathrm{i}k|\bx|}}{|\bx|}\left\{\bQ^\infty(\hat{\bx},\bd,\bp) +\mathcal{O}\big(\frac{1}{|\bx|^2}\big) \right\}\quad \mbox{as}\quad |\bx|\rightarrow\infty, \quad\bQ=\bE,\bH
\end{align*}
uniformly in all directions $\hat{\bx}:=\frac{\bx}{|\bx|}\in \mathbb{S}^2$. $\bE^{\infty}(\hat{\bx},\bd,\bp)$ and $\bH^{\infty}(\hat{\bx},\bd,\bp)$ are known as the electric and magnetic far-field patterns,
corresponding to $\bE^{s}$ and $\bH^{s}$ respectively, which are analytic functions on $\mathbb{S}^2$ and admit the relations (cf.\cite{Colton2013})
$$
\bH^{\infty}(\hat{\bx},\bd,\bp)=\hat{\bx}\times\bE^{\infty}(\hat{\bx},\bd,\bp),\,\,\hat{\bx}\cdot\bE^{\infty}(\hat{\bx},\bd,\bp)=0,\,\,\hat{\bx}\cdot\bH^{\infty}(\hat{\bx},\bd,\bp)=0.
$$


\subsection{Motivation and background}
This paper aims to investigate the time-harmonic electromagnetic (EM) scattering problem in a general medium, which may exhibit a nontrivial topological structure, motivated by two challenging inverse EM scattering problems. Here, we introduce the first inverse EM scattering problem associated with \eqref{eq:sca1}, which involves recovering the properties of a complex scatterer, denoted by $D\oplus(\Omega\backslash\overline{D}; \mu_r, \varepsilon_r)$, using far-field data $\mathbf{E}^{\infty}$ and $\mathbf{H}^{\infty}$. The formulation of this inverse problem is presented as follows:
\begin{equation}\label{eq:ip0 intr}
\bE^\infty\big(\hat{\bx}; \bE^i,\bH^i,\,\bJ, D\oplus(\Omega\backslash\overline{D}; \mu_r, \varepsilon_r) \big)\rightarrow D\oplus(\Omega\backslash\overline{D}; \mu_r, \varepsilon_r).
\end{equation}
It is evident that the inverse problem \eqref{eq:ip0 intr} is nonlinear and ill-posed. Extensive research has been conducted on scenarios involving background media without obstacles, as well as isotropic background media with embedded obstacles; for example, see \cite{Haher1998, Hahenr93, Hettlich1996, Cakoni2004, Colton2013, Cakoni2003,Yang2018,LZY2010}.
However, relatively few results address the inverse electromagnetic scattering problem in an anisotropic medium with an embedded obstacle. The presence of impenetrable obstacles introduces additional challenges, but some progress has been made in this area. Recent studies \cite{LL17, Liu2012, DLL} have focused on similar inverse scattering problems in the context of acoustic and elastic scattering in anisotropic media, aiming to reconstruct both the background medium and its embedded obstacles through multiple measurements. Additionally, \cite{Meng2022} explores time-harmonic elastic scattering in a general anisotropic inhomogeneous medium with an impenetrable obstacle, demonstrating that, by selecting specific material parameters, an isotropic elastic medium can effectively approximate the impenetrable obstacle. These findings align with  \cite{ABCRV2022}, which address the optimization problem using a phase-field approach to reconstruct the cavities.

Another type of inverse EM problem involves the inverse boundary value problem with partial data. Consider the following boundary problem:
\begin{equation}\label{Bp}
\begin{cases}
\nabla\times \bE-\mathrm{i}\omega\mu\bH={\bf 0}&\mbox{in} \quad\Sigma,\\
\nabla\times \bH+(\mathrm{i}\omega\varepsilon-\sigma)\bE={\bf 0}&\mbox{in}\quad \Sigma,\\
\bnu\times\bH=\ba &\mbox{on}\,\,\, \partial \Sigma
\end{cases}
\end{equation}
for $\bE,\bH\in H^1(\mathrm{curl},\Sigma)$, where $\Sigma\Subset \mathbb{R}^3$ is a bounded Lipschitz domain and $\bnu$ is the exterior unit normal vector to $\partial \Sigma$. The problem (\ref{Bp}) exhibits a unique solution $(\mathbf{E},\mathbf{H})\in H^1(\mathrm{curl},\Sigma)$ for $\omega>0$ outside a discrete set of resonant frequencies. The impedance map $\Lambda _{\sigma,\varepsilon,\mu}$ is defined as follows:
$$
\Lambda _{\sigma,\varepsilon,\mu}:\ba\longmapsto\bnu\times \bE\big|_{\partial\Sigma},
$$
which encodes all possible Cauchy data $(\Lambda _{\sigma,\varepsilon,\mu}(\mathbf{\nu}\times\mathbf{H}),\mathbf{\nu}\times\mathbf{H})$ associated with the problem \eqref{Bp}. The inverse problem of interest is formulated as:
\begin{equation}\label{ibp}
\Lambda _{\sigma,\varepsilon,\mu}\longrightarrow (\Sigma;\sigma,\varepsilon,\mu).
\end{equation}
This inverse problem is nonlinear and ill-conditioned, and it has been extensively investigated; see e.g., \cite{COS2009,SU1992} and the related literature cited therein. In practical scenarios, the measurement of data on all boundaries can present challenges, especially when certain boundaries are inaccessible. This limitation is particularly evident in boundary value problems that involve embedded obstacles or holes, where direct measurements of the inner boundary of the medium are not feasible. Therefore, the inverse boundary problem by knowledge of partial boundary measurements is introduced
\begin{equation}\label{ibp'}
\Lambda^p _{\sigma,\varepsilon,\mu}\big|_\Gamma\longrightarrow (\Sigma;\sigma,\varepsilon,\mu),
\end{equation}
where $\Lambda^p_{\sigma,\varepsilon,\mu}:\ba\big|_\Gamma\longmapsto\bnu\times \bE\big|_{\Gamma}$ represents the partial impedance map associated with measurements on a subset $\Gamma\Subset\partial\Omega$. The inverse boundary problems of buried obstacles belong to a specific category of partial-data inverse boundary problems. In this case, we have $\Sigma=\Omega\backslash\overline{D}$ and $\Gamma=\partial \Omega$.
The partial-data inverse problem is widely recognized as an exceptionally challenging research area, and its resolution remains largely elusive, even for the classical Calder\'{o}n inverse conductivity problem $\nabla\cdot(\sigma\nabla u)=0$ (\cite{IUY10, KSU}), where $\sigma$ is a scalar function. This holds true, particularly for the case where $\Sigma=\Omega\backslash\overline{D}$ and $\mathrm{supp}(\ba)\Subset \partial \Omega$, as in \eqref{Bp}. Notable progress has been made in recent years to address the partial-data inverse EM  boundary value problem in both isotropic mediums \cite{COS2009, BMR2014} and anisotropic mediums \cite{KMU2011}. It is important to point out that the material parameters considered for the anisotropic case are sufficiently smooth.

The relationship between the two inverse problems \eqref{eq:sca1} and \eqref{Bp} becomes apparent when we consider the scenario where $\Sigma=\Omega\backslash\overline{D}$ and $\mathrm{supp}(\ba)\Subset \partial \Omega$. In such cases, by introducing  an appropriate truncation to confine the unbounded domain $\mathbb{R}^3\backslash\overline{D}$ into a bounded one, these two problems can be regarded as equivalent. It is worth highlighting that the scattering model \eqref{eq:sca1} is more relevant in the context of the inverse EM scattering problem we are interested in. The mathematical argument associated with an effective medium theory for \eqref{eq:sca1} is technically more involved than associated with  \eqref{Bp}. In this paper, our main focus is devoted to investigating the effective medium theory specifically for \eqref{eq:sca1}.

\subsection{Technical development and discussion}\label{sub2}
In this study, we consider the surrounding medium to be anisotropic and nonhomogeneous, allowing the electric permittivity and conductivity of the underlying medium to be matrix-valued functions with respect to the spatial variable $\mathbf{x}$, while maintaining constant and positive values for those of the medium exterior. Although a positive constant is typically required for the magnetic permeability in cases involving anisotropic electromagnetic materials, we permit it to be a scalar function of the spatial variable $\bx$. Constructing and analyzing a fundamental solution to Maxwell's equations for anisotropic media can be challenging, especially in the presence of impenetrable obstacles. Consequently, methods utilizing fundamental solutions for Maxwell's equations are unsuitable for studying the well-posedness of the scattering problem of multi-layered electromagnetic media with embedded obstacles. In this paper, we propose a different perspective for simultaneously recovering the buried obstacles and the surrounding medium by exploring the effective medium theory. Following a similar approach used in \cite{Meng2022}, we introduce a definition of an effective $\delta$-realization of the complex scatterer, which includes the impenetrable obstacle, the surrounding anisotropic medium, and their corresponding electric and magnetic parameters. Similar to \cite{Meng2022}, we demonstrate that an isotropic electromagnetic medium can effectively approximate the impenetrable obstacle for two different boundary conditions by selecting specific parameters. Our results remain applicable even when considering a more complex obstacle $D$, comprising multiple pairwise disjoint PMC or PEC obstacles.
Transitioning from elastic effective medium theory to EM effective medium theory is not a straightforward generalization. Transitioning from elastic effective medium theory to EM effective medium theory is not a straightforward generalization. The primary challenge arises from the significant disparity in the Sobolev space compact embedding requirements for analyzing electromagnetic scattering problems, which exhibit relatively constrained properties; further details are provided in Subsection \ref{Preliminary}.

As the far-field patterns of electric and magnetic fields can be represented interchangeably, we first focus exclusively on the inverse problem \eqref{eq:ip0 intr}, which specifically addresses a single PEC or PMC obstacle, a topic that has garnered considerable attention in the field. The standard approach to solving \eqref{eq:ip0 intr} involves transforming it into a minimization problem aimed at minimizing the error between the actual and measured far-field data for a priori  class of admissible scatterers. Our main result, as presented in this paper, reformulates the aforementioned optimization problem into the following form:
\begin{equation*}
\min_{(\Omega; \widetilde{\mu}_r, \widetilde{\varepsilon}_r)\in\mathscr{Q}}\big\|\bE^\infty\big(\hat{\bx}; \bE^i,\bJ, (\Omega; \widetilde{\mu}_r, \widetilde{\varepsilon}_r) \big)-\mathcal{F}(D\oplus(\Omega\backslash\overline{D}; \mu_r, \varepsilon_r)) \big\|_{L(\mathbb{S}^{2})^3}.
\end{equation*}
Here, $\mathscr{Q}$ represents the a-priori class of admissible scatterers, while $\mathcal{F}$ denotes the set of measured far-field data. In such cases, the obstacle $D$ can effectively and approximately be replaced by a penetrable medium with specific isotropic electromagnetic medium parameters.  It is worth noting that our results provide a theoretical basis for the rationality of transforming the original problem into the mentioned optimization problem. By solving this optimization problem, we can obtain an initial estimate of the effective realization medium $\Omega$ to solve the original optimization problem.
Furthermore, the distinct singular behaviors exhibited by the material parameters enable us to identify the type of obstacles involved. Specifically, we can analyze these behaviors to determine whether the obstacle $D$ consists of multiple pairwise non-intersecting PEC or PMC obstacles. In such cases, we retain confidence that the previously described procedure remains applicable. However, it is important to note that the processing required in these situations is relatively more complex compared to dealing with a single obstacle.

The rest of the paper is organized as follows. In Section \ref{Main result}, we mainly give the main theorems. Section \ref{sect:preliminary} recalls some significant spaces and operators, and gives auxiliary results. The related results about the effective medium scattering problems are established in Section \ref{sect:medium}. Finally, in Section \ref{sect:proof}, we provide the proofs of Theorem \ref{thm:main1} and Theorem \ref{thm:main2}.

\section{Statement of main results}\label{Main result}
\subsection{Main result for  the anisotropic electromagnetic scattering problem with a single  embedded obstacle}\label{sub:single}
In this subsection, we first consider a simple scenario involving only one obstacle. The following subsection \ref{sub:multiple}, shall provide a more detailed discussion of the situation when there are multiple obstacles.
This subsection aims to present a significant theorem that establishes the existence of approximate effective realizations for a single embedded obstacle in the time-harmonic electromagnetic problem \eqref{eq:sca1}. Based on the discussion in subsection \ref{sub:MS}, it is observed that the far-field patterns of the electric and magnetic fields can be represented interchangeably. Our primary focus is deriving specific results for the associated far-field pattern $\bE^\infty$. Similarly, we can obtain analogous results for the magnetic case using a similar approach. To end this, we reformulate the system (\ref{eq:sca1}) by eliminating the magnetic fields $\bH$ and $\bH^{s}$ as follows:
\begin{equation*}
\begin{cases}
\nabla\times(\mu^{-1}_r\nabla\times \bE)-k^2\varepsilon_r\bE={\bf 0}&\mbox{in}\ \ \Omega\backslash\overline{D},\medskip\\
\nabla\times\nabla\times \bE^s-k^2\bE^s=\bJ& \mbox{in}\ \ \mathbb{R}^3\backslash\overline{\Omega},\medskip\\
\bnu\times \bE\big|_{\partial \Omega}=\bnu\times \bE^s\big|_{\partial \Omega}+\bnu\times\bE^i\big|_{\partial \Omega} &\mbox{on}\ \  \partial{\Omega},\\
\bnu\times (\mu^{-1}_r\nabla\times\bE)\big|_{\partial \Omega}=\bnu\times(\nabla\times \bE^s)\big|_{\partial \Omega}+\bnu\times(\nabla\times\bE^i)\big|_{\partial \Omega} &\mbox{on}\ \  \partial{\Omega},\medskip\\
\mathfrak{B}(\bE)={\bf0} & \mbox{on}\ \ \partial D,\medskip\\
\lim\limits_{r \to\infty}r\big( \nabla\times\bE^s-\mathrm{i}k\bE^s \big)={\bf 0},
\end{cases}
\end{equation*}
where $\mathfrak{B}(\bE)=\bnu\times \bE$ or $\mathfrak{B}(\bE)=\bnu\times\big(\mu^{-1}_r\nabla\times \bE\big)=\bf{0}$, corresponding to whether the obstacle $D$ is a PEC obstacle or a PMC obstacle. In this way, the inverse problem \eqref{eq:ip0 intr} reduces into
\begin{equation}\label{eq:ip1}
\bE^\infty\big(\hat{\bx}; \bE^i,\bJ, D\oplus(\Omega\backslash\overline{D}; \mu_r, \varepsilon_r) \big)\rightarrow D\oplus(\Omega\backslash\overline{D}; \mu_r, \varepsilon_r).
\end{equation}

It is obvious that the inverse problem \eqref{eq:ip1} is nonlinear and ill-posed. Since the presence of the obstacle $D$ inside the scatterer $\Omega$ poses practical challenges in solving \eqref{eq:ip1}. To overcome this issue, we introduce the following definition, which outlines an effective approximation of $D\oplus(\Omega\backslash\overline{D}; \mu_r, \varepsilon_r)$ by utilizing a penetrable scatterer $D$ with specific material parameters.



\begin{defn}\label{def:2}
Consider the complex scatterer $D\oplus(\Omega\backslash\overline{D}; \mu_r, \varepsilon_r)$ as described above. Let $\bE^\infty\big(\hat{\bx}; \bE^i, \bJ, D\oplus(\Omega\backslash
\overline{D}; \mu_r, \varepsilon_r) \big)$ be the far-field pattern corresponding to \eqref{eq:sca1}. 
For sufficiently small $\delta \in \mathbb{R}_+$, suppose there exists a medium configuration denoted by $(\Omega; \widetilde{\mu}_r, \widetilde{\varepsilon}_r)$ such that
$$
(\widetilde{\mu}_r, \widetilde{\varepsilon}_r)\big|_{\Omega\backslash\overline{D}}=(\mu_r, \varepsilon_r)\big|_{\Omega\backslash\overline{D}}, \quad (\widetilde{\mu}_r, \widetilde{\varepsilon}_r)\big|_{D}=(\mu_D, \varepsilon_D),
$$
 where  $\varepsilon_D$ and $\mu_D$  are relative electronic permittivity and magnetic permeability in $D$, respectively. Additionally, consider the corresponding far-field pattern $\bE^\infty(\hat{\bx}; \bE^i, \bJ, (\Omega; \widetilde{\mu}_r, \widetilde{\varepsilon}_r) )$ of the medium scattering problem \eqref{MD:tran2-2} associated with the medium $(\Omega; \widetilde{\mu}_r, \widetilde{\varepsilon}_r)$. If this far-field pattern satisfies the following inequality:
\begin{equation*}
\begin{split}
&\left\|\bE^\infty\left(\hat{\bx}; \bE^i, \bJ, (\Omega; \widetilde{\mu}_r, \widetilde{\varepsilon}_r) \right)-\bE^\infty\left(\hat{\bx}; \bE^i, \bJ, D\oplus(\Omega\backslash\overline{D}; \mu_r, \varepsilon_r)\right ) \right\|_{L(\mathbb{S}^{2})^3} \\
 &\leq C \delta\left(\|\bE^i\|_{H^1(\mathrm{curl},B_R)}+\|\nabla\times\bE^i\|_{H^1(\mathrm{curl},B_R)}+\|\bJ\|_{L^2(B_R)^3}\right),
\end{split}
\end{equation*}
where $B_R$ is any given central ball containing $\Omega$ and $C$ is a generic positive constant depending on the prior parameters, then $(\Omega; \widetilde{\mu}_r, \widetilde{\varepsilon}_r)$ is considered to be an effective $\delta$-realization of $D\oplus(\Omega\backslash\overline{D}; \mu_r, \varepsilon_r)$. Furthermore, if $\delta\equiv 0$, then $(\Omega; \widetilde{\mu}_r, \widetilde{\varepsilon}_r)$ is considered an effective realization of $D\oplus(\Omega\backslash\overline{D}; \mu_r, \varepsilon_r)$. In brief, $(D; \widetilde{\mu}_r, \widetilde{\varepsilon}_r)$ is also referred to as an effective $\delta$-realization of the obstacle $D$.
\end{defn}


\begin{rem}
Here $\bE^\infty(\hat{\bx}; \bE^i, \bJ, (\Omega; \widetilde{\mu}_r, \widetilde{\varepsilon}_r) )$
is the far-field pattern associated with the scattering problem \eqref{eq:sca1}, where the impenetrable obstacle $D\Subset \Omega$ is replaced by a penetrable scatterer $D$  with the medium  configuration of  $(\widetilde{\mu}_r,\widetilde{\varepsilon}_r)\big|_{{D}}$ and the corresponding configuration of  $\Omega \backslash \overline D $ is described  by $(\mu_r, \varepsilon_r)\big|_{\Omega\backslash\overline{D}}$. If $\delta\equiv 0$, then $(\Omega; \widetilde{\mu}_r, \widetilde{\varepsilon}_r)$ is said to be an effective realization of $D\oplus(\Omega\backslash\overline{D}; \mu_r, \varepsilon_r)$.
\end{rem}

In the upcoming theorem, we present our main result concerning the scattering system \eqref{eq:sca1}, with a specific emphasis on the situation involving one single obstacle.
\begin{thm}\label{thm:main1}
Let $D\oplus(\Omega\backslash\overline{D}; \mu_r, \varepsilon_r)$ be the complex scatterer mentioned above, where $\mu_r(\bx)$ satisfies the conditions in \eqref{ineq:gamma} and $\varepsilon_r(\bx)$  is a complex matrix-valued function whose real part and imaginary parts  $\Re\varepsilon_r$ and $\Im\varepsilon_r$ satisfy the conditions in \eqref{ineq:gamma}. These parameters $\mu_r$ and $\varepsilon_r$ can be extended to $\mathbb{R}^3\backslash\overline{D}$ by setting $\mu_r(\bx)=1$ and $\varepsilon_r(\bx)=I$ in $\mathbb R^3\backslash \overline \Omega$, where $I$ is the $3\times 3$ identity matrix. Considering the medium scattering problem \eqref{MD:tran2-2} associated with the medium scatterer $(\Omega; \widetilde{\mu}_r, \widetilde{\varepsilon}_r)$, where $(\widetilde{\mu}_r, \widetilde{\varepsilon}_r)\big|_{\Omega\backslash\overline{D}}=(\mu_r, \varepsilon_r)\big|_{\Omega\backslash\overline{D}}$ and $(\widetilde{\mu}_r, \widetilde{\varepsilon}_r)\big|_{D}=(\mu_D, \varepsilon_D)$. Here $\varepsilon_D$ and $\mu_D$  are the relative electronic permittivity and magnetic permeability respectively. Additionally, let $\eta_0$ and $\tau_0$ be positive constants. For $\delta\in\mathbb{R}_+$ and $\delta\ll 1$,
\begin{itemize}
	\item Case 1: If the obstacle $D$ is a PMC obstacle, then $\mu_D$ and $\varepsilon_D$  can be chosen as
\begin{equation}\label{eq:eff2}
\mu_D=\delta^{-1}\,\,\mbox{and} \,\, \varepsilon_D=\eta_0+\mathrm{i}\tau_0;
\end{equation}
such that $(D; \mu_D, \varepsilon_D)$ is an $\delta^{1/2}$-realization of the obstacle $D$ in the sense of Definition~\ref{def:2}.
\item  Case 2: If the obstacle $D$ is a PEC obstacle, then $\mu_D$ and $\varepsilon_D$ can be taken  as
\begin{equation}\label{eq:eff2_2}
\mu_D=\delta\,\,\mbox{and} \,\, \varepsilon_D=\eta_0+\mathrm{i}\delta^{-1}\tau_0,
\end{equation}
such that $(D; \mu_D, \varepsilon_D)$ is an $\delta^{1/2}$-realization of the obstacle $D$ in the sense of Definition~\ref{def:2}.
\end{itemize}
\end{thm}

\subsection{Main result for the anisotropic electromagnetic scattering problem with multiple embedded obstacles }\label{sub:multiple}
This subsection mainly explores a more complex configuration of $D$ in the electromagnetic scattering problem \eqref{eq:sca1}, where
\begin{equation}\label{eq:M}
D=\bigcup_{l=1}^{N} {D_l},\quad l=1,2,\cdots, N.
\end{equation}
Here, $\big\{D_l\big\}_{l=1}^N$ represents the PMC or PEC obstacles. We assume that for any two distinct components $D_k$ and $D_l$, they satisfy $D_k\cap D_l=\emptyset$ with $k,l=1,2,\cdots,N$, and $\Omega\backslash\overline{D}$ is connected. It is evident that $\partial D\Subset\bigcup_{k=1}^{N}\partial D_k$. Consequently, the electromagnetic scattering problem \eqref{eq:sca1} can be reformulated as follows:
\begin{equation}\label{eq:sca1M}
\begin{cases}
\nabla\times \bE-\mathrm{i}k\mu_r\bH={\bf 0},\quad \nabla\times \bH+\mathrm{i}k\varepsilon_r\bE={\bf0}&\mbox{in}\ \ \Omega\backslash\overline{D},\\
\nabla\times \bE^s-\mathrm{i}k\bH^s={\bf 0},\,\quad \nabla\times \bH^s+\mathrm{i}k\bE^s={\bJ}& \mbox{in}\ \ \mathbb{R}^3\backslash\overline{\Omega},\\
\bnu\times \bE\big|_{\partial \Omega}=\bnu\times \bE^s\big|_{\partial \Omega}+\bnu\times \bE^{i}\big|_{\partial \Omega}&\mbox{on}\ \ \partial{\Omega},\\
\bnu\times \bH\big|_{\partial \Omega}=\bnu\times \bH^s\big|_{\partial \Omega}+\bnu\times \bH^{i}\big|_{\partial \Omega}&\mbox{on}\ \ \partial{\Omega},\medskip\\
\mathfrak{B}(\bE,\bH)\big|_{\partial D_l\cap\partial D}={\bf0} & l=1,2,\cdots,N,\medskip\\
\lim\limits_{r \to\infty}r\big(\bE^s \times\hat{\bx}+\bH^s\big)={\bf 0}.
\end{cases}
\end{equation}
Similarly, eliminating the magnetic fields $\bH$ and $\bH^{s}$, we obtain
\begin{equation*}
\begin{cases}
\nabla\times(\mu^{-1}_r\nabla\times \bE)-k^2\varepsilon_r\bE={\bf 0}&\mbox{in}\ \ \Omega\backslash\overline{D},\\
\nabla\times\nabla\times \bE^s-k^2\bE^s=\bJ& \mbox{in}\ \ \mathbb{R}^3\backslash\overline{\Omega},\\
\bnu\times \bE\big|_{\partial \Omega}=\bnu\times \bE^s\big|_{\partial \Omega}+\bnu\times\bE^i\big|_{\partial \Omega} &\mbox{on}\ \  \partial{\Omega},\medskip\\
\bnu\times (\mu^{-1}_r\nabla\times\bE)\big|_{\partial \Omega}=\bnu\times(\nabla\times \bE^s)\big|_{\partial \Omega}+\bnu\times(\nabla\times\bE^i)\big|_{\partial \Omega} &\mbox{on}\ \  \partial{\Omega},\medskip\\
\mathfrak{B}(\bE)\big|_{\partial D_l\cap\partial D}={\bf0}, & l=1,2,\cdots,N, \medskip\\
\lim\limits_{r \to\infty}r\big( \nabla\times\bE^s-\mathrm{i}k\bE^s \big)={\bf 0}.
\end{cases}
\end{equation*}

We continue to employ the notation $D\oplus(\Omega\backslash\overline{D}; \mu_r, \varepsilon_r)$ to represent a complex scatterer composed of multiple obstacles and their associated material parameters. Moreover, Definition \ref{def:2} can be extended to scenarios that involve multiple obstacles, utilizing the expressions in \eqref{eq:M}. To accommodate such scenarios, the corresponding theorem necessitates appropriate modifications, which are presented in the revised version below.

\begin{thm}\label{thm:main2}

Under the same setup as in Theorem \ref{thm:main1}, let us assume that the parameters $\mu_r$ and $\varepsilon_r$ can be extended to $\mathbb{R}^3\backslash\overline{D}$ by setting $\mu_r(\bx)=1$ and $\varepsilon_r(\bx)=I$ in $\mathbb R^3\backslash \overline{\Omega}$. We consider the medium scattering problem \eqref{MD:tran2-2} associated with the medium scatterer $(\Omega; \widetilde{\mu}_r, \widetilde{\varepsilon}_r)$, where $(\widetilde{\mu}_r, \widetilde{\varepsilon}_r)\big|_{\Omega\backslash\overline{D}}=(\mu_r, \varepsilon_r)\big|_{\Omega\backslash\overline{D}}$ and $(\widetilde{\mu}_r, \widetilde{\varepsilon}_r)\big|_{D}=(\mu_D, \varepsilon_D):=\Big(\sum\limits_{l=1}^N\chi(D_l)\mu_{D_l}, \,\sum\limits_{l=1}^N\chi(D_l)\varepsilon_{D_l}\Big)$, where $\chi$ denotes the characteristic function. If the material parameters $\mu_{D_l}$ and $\varepsilon_{D_l}$ are chosen as in \eqref{eq:eff2} (or \eqref{eq:eff2_2}) for positive constants $\eta^l_0$ and $\tau^l_0$, $l=1,2,\cdots,N$, with $\delta$ being a sufficiently small positive constant, then $(D; \mu_D, \varepsilon_D)$ can be considered as a $\delta^{1/2}$-realization of the complex PMC obstacles $D$ (or PEC obstacles $D$) in the sense of Definition~\ref{def:2}.

\end{thm}
The detailed proofs of Theorem \ref{thm:main1} and Theorem \ref{thm:main2} are postponed to Section \ref{sect:proof}.
\section{Auxiliary results}\label{sect:preliminary}
\subsection{Preliminaries}\label{Preliminary}


To precisely describe the mathematical problem addressed throughout this paper, it is necessary to recall the following spaces and operators.  For $s\in \mathbb{R}$, $H^s(D)^3$ and $H^{s}(\partial D)^3$ represent the standard vector Sobolev spaces defined on $D$ and $\partial D$, respectively (with the convention $H^0=L^2$). Here, $D$ represents a Lipschitz domain in $\mathbb R^3$ with $\partial D$ denoting its boundary. Additionally, other vector function spaces are defined as follows:
\begin{align*}
H^1(\mathrm{curl},D):&=\{\bu\in L^2(D)^3|\mathrm{curl}\bu\in  L^2(D)^3\},\\
H^s_t(\partial D)^3:&=\{\bu\in H^s(\partial D)^3|\bnu\cdot\bu=0\},
\end{align*}
where $\mathrm{curl}$ represents the curl operator in the distributional sense, and $\bnu$ denotes the outward unit normal vector to $\partial D$. Similar spaces are defined for other domains in a similar manner. For any ball $B_R$ centered at the origin with radius $R$, we denote these spaces as $H^1(\mathrm{curl}, B_R \cap\mathbb{R}^3\backslash\overline{D})$ by $H^1_{\mathrm{loc}}(\mathbb{R}^3\backslash\overline{D})$. The tangential trace spaces of $H^1(\mathrm{curl}, D)$ are characterized as follows:
\begin{align*}
H^{-1/2}_{\mathrm{div}}(\partial D):&=\big\{\bg\in H^{-1/2}_t(\partial D)^3 \big|\mathrm{div}_{\small{\partial D}}\,\bg\in H^{-1/2}(\partial D)\big\},\\
H^{-1/2}_{\mathrm{curl}}(\partial D):&=\big\{\bg\in H^{-1/2}_t(\partial D)^3 \big|\mathrm{curl}_{\small{\partial D}}\,\bg\in H^{-1/2}(\partial D)\big\},
\end{align*}
where $\mathrm{div}_{\partial D}$ and $\mathrm{curl}_{\partial D}$ represent the surface divergence and surface curl on $\partial D$, respectively. Next, we recall the tangential trace mapping $\gamma_t$ and the tangential projection operator $\gamma_T$, respectively,
\begin{align*}
\gamma_t(\bu):=\bnu\times\bu
\qquad\mbox{and} \qquad\gamma_T(\bu):=\bnu\times(\bu\times\bnu)
\end{align*}
for any smooth function $C^\infty(\overline{D})^3$, where $\bnu$ is the outward unit  vector to $\partial D$.
In this paper, we usually use $(\cdot)_T$ to represent $\gamma_T(\cdot)$.

%


For a general Lipschitz domain, the tangential trace mapping $\gamma_t$ and the tangential projection operator $\gamma_T$ can be extended from $H^1(\mathrm{curl}, D)$ to $H^{-1/2}{\mathrm{div}}(\partial D)$ and $H^{-1/2}{\mathrm{curl}}(\partial D)$, respectively. Moreover, it is found that the spaces $H^{-1/2}{\mathrm{div}}(\partial D)$ and $H^{-1/2}{\mathrm{curl}}(\partial D)$ can be characterized more precisely (cf. \cite{Monk2002, Buffa2002}). The next theorem primarily presents the extended properties of $\gamma_r$ and $\gamma_T$ and Green's formula associated  with the spaces $H^1(D)^3$ and $H^1(\mathrm{curl}, D)$, corresponding to Theorem 3.29, Theorem 3.31, and Remark 3.30 in \cite{Monk2002}.

\begin{thm}\label{th_auxility1}
Let $D$ be a bounded Lipschitz domain in $\mathbb{R}^3$. Then $\gamma_t$ and $\gamma_T$ can be extended continuously to continuous, linear and surjective maps from $H^1(\mathrm{curl}, D)$ into $H^{-1/2}_{\mathrm{div}}(\partial D)$ and $H^{-1/2}_{\mathrm{curl}}(\partial D)$, respectively. The two spaces $H^{-1/2}_{\mathrm{div}}(\partial D)$ and  $H^{-1/2}_{\mathrm{curl}}(\partial D)$ are dual.
Moreover, the following Green's formulas  hold:
\begin{itemize}
\item[{\rm (1)}] For any $\bv\in H^1(\mathrm{curl}, D)$ and $\bpsi\in H^1(D)^3$,
\begin{equation*}
(\nabla\times \bw,\bpsi)-(\bw,\nabla\times\bpsi)=<\gamma_t(\bw),\bpsi>.
\end{equation*}
\item[{\rm (2)}]For any $\bw\in H^1(\mathrm{curl},D)$ and $\bpsi\in H^1(\mathrm{curl},D)$,
\begin{equation*}
(\nabla\times \bw,\bpsi)-(\bw,\nabla\times\bpsi)=<\gamma_t(\bw),\bpsi_T>.
\end{equation*}

\end{itemize}
\end{thm}

\subsection{Auxiliary lemmas for Case 1 of Theorem \ref{thm:main1}}
In this subsection, we shall establish several key lemmas for proving Case 1 in Theorem \ref{thm:main1}. Subsequently, we consider the following scattering problems: Given $\bq\in H^{-1/2}_{\mathrm{div}}(\partial D)$, $\bff\in H^{-1/2}_{\mathrm{div}}(\partial \Omega) $, $\bh\in H^{-1/2}_{\mathrm{div}}(\partial \Omega)$ and $\bJ $ with $\supp(\bJ)\Subset B_{R_0}\backslash\overline{\Omega}\Subset B_{R}\backslash\overline{\Omega}$, find $(\bE,\bH,\bE^s,\bH^s)\in H^1(\mathrm{curl}, \Omega\backslash\overline{D})\times H^1(\mathrm{curl}, \Omega\backslash\overline{D})\times  H^1_{\mathrm{loc}}(\mathrm{curl}, \mathbb{R}^3\backslash\overline{\Omega})\times H^1_{\mathrm{loc}}(\mathrm{curl}, \mathbb{R}^3\backslash\overline{\Omega})$ satisfying
\begin{align}\label{MD:tran2}
\left\{ \begin{array}{ll}
\nabla\times\bE-\mathrm{i}k\mu_r\bH={\bf0},\quad \nabla\times\bH+\mathrm{i}k\varepsilon_r\bE={\bf0}&\mbox{in}\ \  \Omega\backslash \overline{D},\quad \\[5pt]
\nabla\times\bE^s-\mathrm{i}k\bH^s={\bf0},\quad \nabla\times\bH^s+\mathrm{i}k\bE^s={\bJ}& \mbox{in}\ \ \mathbb{R}^3\backslash\overline{\Omega},\\[5pt]
\bnu\times\bE-\bnu\times\bE^s=\bff,\quad \bnu\times\bH-\bnu\times\bH^s=\bh&\mbox{on} \ \ \partial{\Omega},\\[5pt]
\bnu\times\bH=\bq&\mbox{on}\ \ \partial{D},\\[5pt]
\lim\limits_{|\bx| \to\infty}|\bx|\big(\bE^s \times\hat{\bx}+\bH^s\big)={\bf 0}.
\end{array}
 \right.
\end{align}



To formulate \eqref{MD:tran2} in a variational  form over a bounded domain, we introduce an artificial boundary $\partial B_R$, representing the surface of a ball $B_R$ with radius $R$, ensuring that the scatterer is contained within the ball's interior. Moreover, given $\bff\in H^{-1/2}{\mathrm{div}}(\partial \Omega)$, if $k^2$ is not a Dirichlet eigenvalue of the operator \lq\lq  $\nabla \times \nabla \times $ " in $B_R\backslash \Omega $,  there exists a unique solution $\bE_{\bff}\in H^1(\mathrm{curl}, B_R\backslash\overline{\Omega})$ to the following equation
 \begin{align}\label{eq:Ef}
&\nabla\times\nabla\times \bw-k^2\bw={\bf0},\qquad \nabla\cdot \bw=0\,\,\, \mbox{in}\,\, B_R\backslash \overline{\Omega},\nonumber\\
 &\bnu\times\bw=\bff\,\quad \mbox{on}\,\quad\partial \Omega,\qquad \bnu\times\bw={\bf0}\quad \mbox{on}\quad\partial D,
\end{align}
and it can be estimated as follows
\begin{equation}\label{ieq:Ef}
\|\bE_{\bff}\|_{H^1(\mathrm{curl},B_R\backslash\overline{\Omega})}\leq C\|\bff\|_{H^{-1/2}_{\mathrm{div}}(\partial D)}\quad \mbox{for some constant $C>0$}.
\end{equation}
 Related uniqueness results can be referred to \cite{LZY2010} and the references cited therein. Now we introduce the exterior Calder\'{o}n operator $G_e$ (cf. \cite{Kirsch1998,Monk2002}), which is an analogue of the Dirichlet-to-Neumann (DtN) map for Maxwell equations. Essentially, $G_e$ is an isomorphism mapping from $H^{-1/2}{\mathrm{div}}(\partial B_R)$ to $H^{-1/2}{\mathrm{div}}(\partial B_R)$, namely,
\begin{align}\label{eq:Calderon}
G_e:&H^{-1/2}_{\mathrm{div}}(\partial B_R)\rightarrow H^{-1/2}_{\mathrm{div}}(\partial B_R)\nonumber\\
&\blambda\longmapsto G_e(\blambda):=\hat{\bx}\times\bH^s \quad \mbox{on}\,\,\partial B_R,
\end{align}
where $(\bE^s, \bH^s)$ satisfies
\begin{align*}
\left\{ \begin{array}{ll}
\nabla\times\bE^s-\mathrm{i}k\bH^s={\bf0}&\mbox{in}\ \mathbb{R}^3\backslash \overline{B_R},\quad \\[5pt]
\nabla\times\bH^s+\mathrm{i}k\bE^s={\bf0}&\mbox{in}\ \mathbb{R}^3\backslash \overline{B_R},\quad \\[5pt]
\hat{\bx}\times \bE^s=\blambda&\mbox{on}\ \partial B_R,\quad \\[5pt]
\lim\limits_{r \to\infty}r(  \bE^s\times\hat{\bx}+\bH^s  \big)={\bf0}.
\end{array}
\right.
\end{align*}

By expressing the magnetic fields through electric fields in \eqref{MD:tran2}, and utilizing the transmission conditions across $\partial \Omega$ and the boundary condition on $\partial D$, along with the definition of $G_e$ and integration by parts, we obtain the following variational formulation for the electric fields in \eqref{MD:tran2}. Given $\bq\in H^{-1/2}_{\mathrm{div}}(\partial D)$, $\bff\in H^{-1/2}_{\mathrm{div}}(\partial \Omega)$, $\bh\in H^{-1/2}_{\mathrm{div}}(\partial \Omega)$ and $\bJ $ with $\supp(\bJ)\Subset B_{R_0}\backslash\overline{\Omega}\Subset B_{R}\backslash\overline{\Omega}$, find $\bU\in H^1(\mathrm{curl},B_R\backslash\overline{D})$ satisfying
\begin{small}\begin{align}\label{eq:VF}
&\int_{\Omega\backslash\overline{D}}\mu_r^{-1}(\nabla\times \bU)\cdot(\nabla\times\overline{\bPhi})-k^2\varepsilon_r\bU\cdot\overline{\bPhi}\mathrm{d}\bx+\int_{B_R\backslash\overline{\Omega}}(\nabla\times \bU)\cdot(\nabla\times\overline{\bPhi})-k^2\bU\cdot\overline{\bPhi}\mathrm{d}\bx\nonumber\\
&
+\mathrm{i}k\int_{\partial B_R}G_e(\hat{\bx}\times \bU)\cdot\overline{\bPhi}_{\mathrm{T}}\,\mathrm{d}\sigma=-\int_{\partial \Omega}\mathrm{i}k\bh\cdot\overline{\bPhi}_{\mathrm{T}}\mathrm{d}\sigma+\int_{\partial D}\mathrm{i}k\bq\cdot\overline{\bPhi}_{\mathrm{T}}\,\mathrm{d}\sigma-\mathrm{i}k\int_{\partial B_R}G_e(\hat{\bx}\times\bE_{\bff})\cdot\overline{\bPhi}_{\mathrm{T}}\,\mathrm{d}\sigma\quad\nonumber\\
&-\int_{B_R\backslash\overline{\Omega}}(\nabla\times \bE_{\bff})\cdot(\nabla\times\overline{\bPhi})-k^2\bE_{\bff}\cdot\overline{\bPhi}\mathrm{d}\bx+\int_{B_R\backslash\overline{\Omega}}\mathrm{i}k \bJ\cdot\overline{\bPhi}\,\mathrm{d}\bx
\end{align}
\end{small}for any test function $\bPhi\in H^1(\mathrm{curl},B_R\backslash\overline{D})$. Assume that $\bU$ is a solution of \eqref{eq:VF}, by choosing sufficiently smooth test functions $\bPhi$, we can easily show that $\bE|_{D}:=\bU|_{D}$ and $\bE^s:=\bU|_{B_R\backslash\overline{\Omega}}+\bE_{\bff}$ satisfy the differential equations for the electric fields of \eqref{MD:tran2} in $\Omega\backslash\overline{D}$ and $B_R\backslash\overline{\Omega}$, respectively, along with  the transmission conditions on $\partial \Omega$ and $\hat{\bx}\times(\nabla\times\bE^s)=\mathrm{i}k G_e(\hat{\bx}\times\bE^s)$ on $\partial B_R$. Furthermore, a solution of the variational problem \eqref{eq:VF} and the corresponding magnetic fields $\bH^s=\frac{\nabla\times \bE^s}{\mathrm{i}k}$ and $\bH=\frac{\nabla\times \bE}{\mathrm{i}k}$ can be extended to a solution of \eqref{MD:tran2}.  Notably,  at the interface $\partial B_R$, there is no jump of $\hat{\bx}\times \bE^s$. Consequently, based on the relationship between $\hat{\bx}\times\bE^s$ and $\hat{\bx}\times\bH^s$ through the operator $G_e$, we infer that  there is no jump of $\hat{\bx}\times\bH^s$ either.

\begin{thm}\label{th:unique}
The problems \eqref{MD:tran2} and \eqref{eq:VF} have at most one solution.
\end{thm}
\begin{proof}
We only need to show uniqueness for the problem \eqref{MD:tran2}. Assume that $\bE$, $\bE^s$, $\bH=1/(\mathrm{i}k\mu_r)\nabla\times \bE$ and $\bH^s=1/(\mathrm{i}k)\nabla\times \bE^s$ are the solutions of \eqref{MD:tran2} with boundary conditions $\bff=\bh=\bq=\bJ={\bf 0}$. Taking the dot product of the first equation in \eqref{MD:tran2} by the electric fields $\overline{\bE}$  in the domain $\Omega\backslash\overline{D}$, we see that
\begin{equation*}
\int_{\Omega\backslash\overline{D}}\left(\nabla\times(\mu_r^{-1}\nabla\times\bE)\right)\cdot\overline{\bE}\,\mathrm{d}\bx-\int_{\Omega\backslash\overline{D}}k^2\varepsilon_r\bE\cdot\overline{\bE}\,\mathrm{d}\bx=\bf 0.
\end{equation*}

Using Green's formula and the fact that $\nabla(\alpha \,\ba)=\nabla\alpha\times\ba+\alpha\nabla\times\ba$, $\ba\cdot(\bb\times\bc)=\bb\cdot(\bc\times\ba)=\bc\cdot(\ba\times\bb),$ $\ba\times\bb=-\bb\times\ba$ for arbitray vectors $\ba$, $\bb$, $\bc$ and scalar $\alpha$, we see that
\begin{equation*}
\int_{\Omega\backslash\overline{D}}\mu_r^{-1}|\nabla\times\bE|^2\mathrm{d}\bx-\int_{\Omega\backslash\overline{D}}k^2\varepsilon_r\bE\cdot\overline{\bE}\,\mathrm{d}\bx-\int_{\partial(\Omega\backslash\overline{D})}\big(\bnu\times(\nabla\times\bE)\big)\cdot(\mu_r^{-1}\overline{\bE})_{\mathrm T}\,\mathrm{d}\sigma=\bf 0.
\end{equation*}
Repeating the above procedure in $B_R\backslash\overline{\Omega}$ with $\bE^s$ replaced by $\bE$, we observe that
{\small\begin{equation*}
\int_{B_R\backslash\overline{\Omega}}|\nabla\times\bE^s|^2\,\mathrm{d}\bx-\int_{B_R\backslash\overline{\Omega}}k^2|\bE^s|^2\,\mathrm{d}\bx-\int_{\partial B_R}\big(\bnu\times(\nabla\times\bE^s)\big)\cdot\overline{\bE}^s_{\mathrm T}\,\mathrm{d}\sigma+\int_{\partial \Omega}\big(\bnu\times(\nabla\times\bE^s)\big)\cdot\overline{\bE}^s_{\mathrm T}\,\mathrm{d}\sigma=\bf 0.
\end{equation*}}
Summing the above two formulae and then using the transmission conditions on $\partial \Omega$ and $\partial B_R$,  we get
{\small\begin{align}\label{interD}
\int_{\partial B_R}(\bnu\times(\nabla\times\bE^s))\cdot\overline{\bE}^s\mathrm{d}\sigma
&=\int_{B_R\backslash\overline{\Omega}}|\nabla\times\bE^s|^2-k^2|\bE^s|^2\mathrm{d}\bx+\int_{\Omega\backslash\overline{D}}\mu^{-1}_r|\nabla\times\bE|^2-k^2\varepsilon_r\bE\cdot\overline{\bE}\,\mathrm{d}\bx
\end{align}}
Using the identity $\mathrm{i}k\bH^s=\nabla\times \bE^s$ in $B_R\backslash\overline{\Omega}$ and the fact that $(\bnu\times \ba)\cdot\bb=-(\bnu\times \bb)\cdot\ba$ again for any vector $\ba$ and $\bb$, it  follows directly that

\begin{align*}
-\mathrm{i}k\int_{\partial B_R}(\bnu\times\overline{\bE}^s)\cdot\bH^s\mathrm{d}\sigma&=\mathrm{i}k\int_{\partial B_R}(\bnu\times\bH^s)\cdot\overline{\bE}^s\mathrm{d}\sigma=\int_{B_R\backslash\overline{\Omega}}|\nabla\times\bE^s|^2-k^2|\bE^s|^2\mathrm{d}\bx\nonumber\\
&\quad+\int_{\Omega\backslash\overline{D}}\mu^{-1}_r|\nabla\times\bE|^2-k^2\varepsilon_r|\bE|^2\mathrm{d}\bx.
\end{align*}
Taking the imaginary part of \eqref{interD}, we imply that
$$
\Re\left(\int_{\partial B_R}(\bnu\times\overline{\bE}^s)\cdot\bH^s\mathrm{d}\sigma\right)=-k\int_{\Omega\backslash\overline{D}}\Im(\varepsilon_r)\bE\cdot\overline{\bE}\mathrm{d}\bx\leq 0.
$$
From the Rellich lemma (cf.\cite[Theorem 6.11]{Colton2013}), we conclude that $\bE^s=\bH^s={\bf 0}$ in $\mathbb{R}^3\backslash\overline{\Omega}$, simplifying the transmission conditions to the continuity of the tangential components of the electric and magnetic fields. Lastly, applying the unique continuation principle, we deduce that $\bE$ and $\bH$ are zero in $\Omega\backslash\overline{D}$.

The proof is complete.
\end{proof}

Next, we will demonstrate the existence of a solution to \eqref{eq:VF}. Define the sesquilinear form
\begin{align*}
\mathcal{M}:H^1(\mathrm{curl},B_R\backslash\overline{D})\times H^1(\mathrm{curl},B_R\backslash\overline{D}) \longrightarrow\mathbb{C}
\end{align*}
and the linear form $\mathcal{F}:H^1(\mathrm{curl},B_R\backslash\overline{D})\longrightarrow\mathbb{C}$ as follows
\begin{align}
&\mathcal{M}(\bU,\bPhi):=\int_{B_R\backslash\overline{D}}\mu_r^{-1}(\nabla\times\bU)\cdot(\nabla\times\overline{\bPhi})-k^2\varepsilon_r\bU\cdot\overline{\bPhi}\,\mathrm{d}\bx+\mathrm{i}k\int_{\partial B_R}G_e(\hat{\bx}\times\bU)\cdot\overline{\bPhi}_{\mathrm{T}}\,\mathrm{d}\sigma\label{eq:A(U,V)}\\
&\mathcal{F}(\bPhi):=-\int_{\partial \Omega}\mathrm{i}k\bh\cdot\overline{\bPhi}_{\mathrm{T}}\,\mathrm{d}\sigma-\mathrm{i}k\int_{\partial B_R}G_e(\hat{\bx}\times\bE_{\bff})\cdot\overline{\bPhi}_{\mathrm{T}}\,\mathrm{d}\sigma+\int_{\partial D}\mathrm{i}k\bq\cdot\overline{\bPhi}_{\mathrm{T}}\mathrm{d}\sigma\nonumber\\
&\qquad\qquad-\int_{B_R\backslash\overline{\Omega}}(\nabla\times \bE_{\bff})\cdot(\nabla\times\overline{\bPhi})-k^2\bE_{\bff}\cdot\overline{\bPhi}\mathrm{d}\bx+\int_{B_R\backslash\overline{\Omega}}\mathrm{i}k \bJ\cdot\overline{\bPhi}\,\mathrm{d}\bx,\label{eq:F(v)}
\end{align}
where
\begin{equation}\label{eq:U_E}
\mu_r(\bx)=
\begin{cases}
\mu_r(\bx)\big|_{\Omega
\backslash \overline D},\quad  \bx\in \Omega \backslash \overline D,\\
{1} ,\hspace{1.9cm}  \bx \in  B_R\backslash \overline \Omega
\end{cases} \,\, \mbox{and}\,\,\, \varepsilon_r(\bx)=
\begin{cases}
\varepsilon_r(\bx)\big|_{\Omega
\backslash \overline D},\quad \bx\in \Omega \backslash \overline D,\\
I,\hspace{1.9cm}  \bx \in  B_R\backslash \overline \Omega.
\end{cases}
\end{equation}
It is evident that $\mathcal{M}(\cdot,\cdot)$ is bilinear, and $\mathcal{F}(\cdot)$ is linear. Recognizing that the embedding operator $H^1(\mathrm{curl},B_R\backslash\overline{D})\rightarrow L^2(B_R\backslash\overline{D})$ is not compact, we will employ the Helmholtz decomposition of the space $H^1(\mathrm{curl},B_R\backslash\overline{D})$ to isolate the null space $\nabla\mathcal{S}$ from $H^1(\mathrm{curl}, B_R\backslash\overline{D})$. To that end, we introduce the following spaces:
\begin{equation}\label{eq:S}
\mathcal{S}:=\left\{\psi\in H^1(B_R\backslash\overline{D})\Big|\int_{\partial B_R}\psi\,\mathrm{d}\sigma=0\right\}
\end{equation}
and
\begin{align}\label{def:X0}
\mathrm{X}_0:&=\bigg\{ \bU\in H^1(\mathrm{curl},B_R\backslash\overline{D})\bigg|-k^2\int_{\Omega\backslash\overline{D}} \varepsilon_r\bU\cdot\nabla\overline{\psi}\,\mathrm{d}\bx-k^2\int_{B_R\backslash\overline{\Omega}}\bU\cdot\nabla\overline{\psi}\,\mathrm{d}\bx\nonumber\\
&\qquad\qquad\qquad\qquad\qquad\qquad+\mathrm{i}k\int_{\partial B_R}G_e(\hat{\bx}\times\bU)\cdot\nabla_{\partial B_R}\overline{\psi}\,\mathrm{d}\sigma=0
\,\,\mbox{for all }\,\, \psi\in \mathcal{S}         \bigg\}\nonumber\\
&=\bigg\{\bU\in H^1(\mathrm{curl},B_R\backslash\overline{D})\bigg|\nabla\cdot(\varepsilon_r\bU)=0\,\, \mbox{in}\,\, B_R\backslash\overline{D},\,\, -k^2\,\hat{\bx}\cdot\bU=\mathrm{i}k\,\nabla_{\partial B_R}\cdot G_e(\hat{\bx}\times\bU)\,\, \nonumber\\
&\qquad\qquad\qquad\qquad\qquad\qquad\qquad\mbox{on}\,\,\partial B_R\quad\mbox{and}\,\, \bnu\cdot\bU=0\,\, \mbox{on}\,\, \partial D                                      \bigg\},
\end{align}
where $\varepsilon_r$ is given by \eqref{eq:U_E} and $\nabla_{\partial B_R}\cdot(\hat{\bx}\times\bU)$ is defined by
\begin{equation}\label{eq:grad_BR}
 \nabla_{\partial B_R}\cdot(\hat{\bx}\times\bU):=-\bnu\cdot(\nabla\times \bU)\big|_{\partial B_R}.
\end{equation}
 It is straightforward to verify that $\mathcal{S}$ is a Hilbert space associated with the norm $\|\cdot\|_{H^1(B_R\backslash\overline{D})^3}$.


Let $\widetilde{G}_e$ be defined as $G_e$ in \eqref{eq:Calderon} but with $k$ replaced by $\mathrm{i}$. Since $\big<\widetilde{G}e\blambda,\blambda\times \hat{\bx}\big><0$ for any $\blambda \in H^{-1/2}{\mathrm{div}}(\partial B_R)$, $\widetilde{G}_e$ is negative definite. The following lemma presents some properties of the Calder'{o}n operators $G_e$ and $\widetilde{G}_e$ (cf. \cite{Monk2002}), which are crucial in proving the auxiliary lemmas.

\begin{lem}\label{lem:proper1}
Let $G_e$ and $\widetilde{G}_e$ be two Calder\'{o}n operators, which are defined as above. Then we have the following properties:
\begin{itemize}
\item[{\rm (1)}] For any non-zero $\blambda \in H^{-1/2}_{\mathrm{div}}(\partial B_R)$, $\Big|\big<\widetilde{G}_e\blambda,\blambda\times \hat{\bx}\big>\Big|\geq c\|\blambda\|^2_{H^{-1/2}_{\mathrm{div}}(\partial B_R)};$
\item[{\rm (2)}] $G_e+\mathrm{i}\,\widetilde{G}_e:H^{-1/2}_{\mathrm{div}}(\partial B_R)\rightarrow H^{-1/2}_{\mathrm{div}}(\partial B_R)$ is compact;
\item[{\rm (3)}] There are two operators $G^1_e$ and $G^2_e$ such that $G_e\blambda:=G^1_e\blambda+G^2_e\blambda$, where
 \begin{itemize}\item[{\rm (i)}]$G^1_e:\bE\longmapsto G^1_e(\hat{\bx}\times\bE)$ is compact from $\mathrm{X}_0$ into $H^{-1/2}_{\mathrm{div}}(\partial B_R)$;
 \item[{\rm (ii)}]$\mathrm{i}k\big<G_e^2(\hat{\bx}\times\bE),\overline{\bE}_{\mathrm{T}}\big>_{H^{-1/2}_{\mathrm{div}}(\partial B_R)\times H^{-1/2}_{\mathrm{div}}(\partial B_R)}\geq 0$ for all $\bE\in H^1(\mathrm{curl}, B_R\backslash\overline{D})$,
\end{itemize}
where $\bE_{\mathrm T}$ represents the tangential projection of $\bE$ onto $\partial B_R$.
\end{itemize}
\end{lem}


To analyze the properties of these operators on $\mathcal{S}$, we now consider the following sesquilinear forms $\mathcal{A}_1(\cdot,\cdot)$ and $\mathcal{A}_2(\cdot,\cdot)$ on $\mathcal{S}\times\mathcal{S}$:
\begin{align*}
\mathcal{A}_1(\nabla\phi,\nabla\psi):&=-k^2\int_{B_R\backslash\overline{D}}\varepsilon_r(\bx)\nabla\phi\cdot\nabla\overline{\psi}\,\mathrm{d}\bx+k^2\int_{\partial B_R}\widetilde{G}_e(\hat{\bx}\times\nabla\phi)\cdot\nabla_{\partial B_R}\overline{\psi}\mathrm{d}\sigma,\\
\mathcal{A}_2(\nabla\phi,\nabla\psi):&=\mathrm{i}k\int_{\partial B_R}(G_e+\mathrm{i}\widetilde{G}_e)(\hat{\bx}\times\nabla\overline{\psi})\mathrm{d}\sigma,
\end{align*}
which satisfies
$$
\mathcal{M}(\nabla\phi,\nabla\psi)=\mathcal{A}_1(\nabla\phi,\nabla\psi)+\mathcal{A}_2(\nabla\phi,\nabla\psi)\quad \mbox{for all}\quad \phi,\psi\in \mathcal{S}.
$$

\begin{lem}\label{lem:compact1}
With the same notation and setup as above, the following holds true
\begin{itemize}
\item[{\rm (1)}] $\mathcal{A}_1(\cdot,\cdot)$ is bounded and coercive on $\mathcal{S}\times\mathcal{S}$ and there exists a compact operator $K_1$ from $\mathcal{S}$ into itself which satisfies $\mathcal{A}_2(\phi,\xi)=\mathcal{A}_1(K_1\phi,\xi)$ for all $\phi$, $\xi\in \mathcal{S}$.
\item[{\rm (2)}]
    The operator $I+K_1$ is an isomorphism from $\mathcal{S}$ onto itself. Furthermore, there exists a unique solution $a\in\mathcal{S}$ to the variational problem $\mathcal{M}(\nabla a,\nabla\phi)=\mathcal{F}(\phi)$ for all $\phi\in \mathcal{S}$ and the solution can be given by $a=(I+K_1)^{-1}b$, where $b$ is in $\mathcal{S}$ and satisfies $\mathcal{A}_1(b,\phi)=\mathcal{F}(\phi)$ for all $\phi\in\mathcal{S}$.
\end{itemize}
\end{lem}
\begin{proof}
The proof follows a similar argument to that presented in \cite[Theorem 10.2]{Monk2002}, thus we omit the details.
\end{proof}

Before further analyzing $\mathrm{X}_0$ and $\nabla\mathcal{S}$, we introduce some significant auxiliary lemmas.
\begin{lem}\label{lem:Helmotz}
The spaces $\nabla\mathcal{S}$ and $\mathrm{X}_0$ are closed subspaces of $H^1(\mathrm{curl},B_R\backslash\overline{D})$ and the space $H^1(\mathrm{curl},B_R\backslash\overline{D})$ is the direct sum of these two subspaces, namely,  $H^1(\mathrm{curl},B_R\backslash\overline{D})=\nabla\mathcal{S}\bigoplus\mathrm{X}_0$. Furthermore, the projections onto the subspaces are bounded, and for all  $\bU$ in $ \mathrm{X}_0$ and $\psi$ in $\mathcal{S}$, there exist positive constants $C_1$ and $C_2$ satisfying
{\small\beq\label{eq:HM}
C_1\|\bU+\nabla\psi\|_{H^1(\mathrm{curl},B_R\backslash\overline{D})}\leq \|\bU\|_{H^1(\mathrm{curl},B_R\backslash\overline{D})}+\|\nabla\psi\|_{H^1(\mathrm{curl},B_R\backslash\overline{D})}
\leq C_2\|\bU+\nabla\psi\|_{H^1(\mathrm{curl},B_R\backslash\overline{D})}.
\eeq}
\end{lem}

\begin{proof}

From the definition of $\mathcal{S}$ in \eqref{eq:S}, it is evident that $\nabla\mathcal{S}$ is closed. Considering that $\bU\longmapsto(\varepsilon_r\bU,\nabla\psi)$ and $\bU\longmapsto\big<G_e(\hat{\bx}\times\bU),\nabla\psi|_{\partial B_R}\big>$ are linear and bounded operators on $H^1(\mathrm{curl},B_R\backslash\overline{D})$, we can directly conclude that $\mathrm{X}_0$ is a closed subspace of $H^1(\mathrm{curl},B_R\backslash\overline{D})$.


We now aim to demonstrate that $H^1(\mathrm{curl},B_R\backslash\overline{D})=\mathrm{X}_0\oplus\nabla\mathcal{S}$. Given a fixed $\bU\in H^1(\mathrm{curl},B_R\backslash\overline{D})$, our objective is to construct $p\in\mathcal{S}$ as the solution to the following equation:
$$
\mathcal{M}(\nabla p,\nabla\xi)=\mathcal{M}(\bU,\nabla\xi)\quad \mbox{for all $\xi\in\mathcal{S}$},
$$
where $\mathcal{M}(\cdot,\cdot)$ is given by \eqref{eq:A(U,V)}.
Owing to Lemma \ref{lem:proper1}, it shows that this problem is well-posed and $p$ has the following estimate
$$
\|\nabla p\|_{L^2(B_R\backslash\overline{D})^3}\leq C\|\bU\|_{H^1(\mathrm{curl},B_R\backslash\overline{D})}\quad\mbox{for a fixed constant $C>0$}.
$$
Let $\bU_0=\bU-\nabla p$. Consequently, $\bU_0\in\mathrm{X}_0$ follows from the definition of $\mathrm{X}_0$ provided by \eqref{def:X0}. Next, our objective is to demonstrate that $\nabla\mathcal{S}\cap\mathrm{X}_0={\bf 0}$. Suppose $\bU=\nabla p\in \nabla\mathcal{S}\cap\mathrm{X}_0$, then we infer that
$$
0=\mathcal{M}(\bU,\nabla\xi)=\mathcal{M}(\nabla p,\nabla\xi) \quad\mbox{for all}\,\,\xi\in\mathcal{S}.
$$
Using Lemma \ref{lem:proper1} again, we have $p=0$ and then $\bU=\bf 0$. Thus, we prove the direct sum $H^1(\mathrm{curl},B_R\backslash\overline{D})=\nabla\mathcal{S}\bigoplus\mathrm{X}_0$.


Due to the boundedness of the projection operators $H^1(\mathrm{curl},B_R\backslash\overline{D})\hookrightarrow\nabla\mathcal{S}$ and $H^1(\mathrm{curl},B_R\backslash\overline{D})\hookrightarrow \mathrm{X}_0$, proving \eqref{eq:HM} becomes straightforward.

The proof is complete.
\end{proof}

\begin{lem}\label{lem:imbed}
The embedding operator from $\mathrm{X}_0$ into $ L^2(B_R\backslash\overline{D})^3$ is compact.
\end{lem}

\begin{proof}
The method used for proving the result is adapted from \cite{Monk2002}. Consider a bounded sequence $\{\bV_j\}_{j=1}^\infty\subset \mathrm{X}_0$ such that $\bV_j\rightarrow \bf 0$ weakly. Our objective is to demonstrate that $\bV_j\rightarrow \bf 0$ as $j\rightarrow\infty$ in $H^1(\mathrm{curl},B_R\backslash\overline{D})$. For this purpose, we extend $\bV_j$ into $\mathbb{R}^3\backslash\overline{B_R}$ by defining $\bV_j=\bW_j$, where $\bW_j\in H^1(\mathrm{curl},\mathbb{R}^3\backslash\overline{B_R})$ satisfies the following exterior Maxwell equation.
\begin{align*}
\left\{ \begin{array}{ll}
\nabla\times\nabla\times \bW_j-k^2\bW_j={\bf0}&\mbox{in}\ \mathbb{R}^3\backslash \overline{B_R},\quad \\[5pt]
\hat{\bx}\times\bW_j=\hat{\bx}\times\bV_j&\mbox{on}\ \partial B_R,\quad \\[5pt]
\lim\limits_{r \to\infty}r\Big(  \hat{\bx}\times (\nabla\times \bW_j)  +\mathrm{i}k \bW_j  \Big)={\bf0}.
\end{array}
\right.
\end{align*}
It follows from the definition of $\mathrm{X}_0$ that $-k^2\,\hat{\bx}\cdot\bV_j=\mathrm{i}k\nabla{\partial B_R}\cdot G_e(\hat{\bx}\times \bV_j)$ on $\partial B_R$. By utilizing the definition of $G_e$ and (3.52) in \cite{Monk2002}, we obtain
{\small\begin{align*}
\hat{\bx}\cdot\bV_j&=\frac{1}{\mathrm{i}k}\nabla_{\partial B_R}\cdot G_e(\hat{\bx}\times\bV_j)=\frac{1}{\mathrm{i}k}\nabla_{\partial B_R}\cdot \big(\frac{1}{\mathrm{i}k}\times\big(\nabla\times \bW_j\big)\big)=-\frac{1}{k^2}\nabla_{\partial B_R}\cdot \big(\hat{\bx}\times(\nabla\times\bW_j)\big)\\
&=\frac{1}{k^2}\hat{\bx}\cdot\big(\nabla\times\nabla\times \bW_j\big)=\hat{\bx}\cdot\bW_j.
\end{align*}}
Since $\nabla \cdot\bV_j=0$ in $B_R\backslash\overline{D}$ and $\mathbb{R}^3\backslash\overline{B_R}$, it is concluded from the equation above that
$$
\nabla \cdot(\varepsilon_r\bV_j)=0\quad \mbox{and}\quad \bnu\cdot\bV_j=0\quad \mbox{on}\quad \partial D.
$$
From Theorem 3.50 in \cite{Monk2002}, we have $\bV_j\in H^{1/2+\delta}_{\mathrm{loc}}(\mathbb{R}^3\backslash\overline{D})^3$ for $\delta>0$. In virtue of the compact imbedding mapping $H^{1/2}(B_R\backslash\overline{D})^3\rightarrow L^2(B_R\backslash\overline{D})^3$, there exists a subsequence of $\{\bV_j\}_{j\geq 1}$ can be extracted to convergence to $\mathbf 0$ in $L^2(B_R\backslash\overline{D})^3$.

The proof is complete.
\end{proof}

Based on the above preparations, we are in a position to provide the existence of the solution to \eqref{eq:VF}, where this solution enables us to establish corresponding estimations for the effective realization as stated in Theorem \ref{thm:main1}.

\begin{thm}\label{th:exitence1}
The scattering problem \eqref{MD:tran2} has a unique solution $(\bE,\bH,\bE^s,\bH^s)\in H^1(\mathrm{curl}, \Omega\backslash\overline{D})\times H^1(\mathrm{curl}, \Omega\backslash\overline{D})\times  H^1_{\mathrm{loc}}(\mathrm{curl}, \mathbb{R}^3\backslash\overline{\Omega})\times H^1_{\mathrm{loc}}(\mathrm{curl}, \mathbb{R}^3\backslash\overline{\Omega})$. Furthermore, one has the following estimate
\begin{align}\label{ieq:est}
&\|\bE\|_{H^1(\mathrm{curl},\Omega\backslash\overline{D})}+\|\bH\|_{H^1(\mathrm{curl},\Omega\backslash\overline{D})}+\|\bE^s\|_{H^1(\mathrm{curl}, B_R\backslash\overline{\Omega})}+\|\bH^s\|_{H^1(\mathrm{curl}, B_R\backslash\overline{\Omega})}\nonumber\\
&\qquad\quad\leq C\Big(\|\bff\|_{H^{-1/2}_{\mathrm{div}}(\partial \Omega)}+\|\bh\|_{H^{-1/2}_{\mathrm{div}}(\partial \Omega)}+\|\bq\|_{H^{-1/2}_{\mathrm{div}}(\partial D)}+\|\bJ\|_{L^2(B_R\backslash\overline{\Omega})^3}
\Big)
\end{align}
\end{thm}
\begin{proof}

For $\bU,\bV\in H^1(\mathrm{curl}, B_R\backslash\overline{D})$, the decompositions $\bU=\bU^0+\nabla p$ and $\bV=\bV^0+\nabla q$ follow Lemma \ref{lem:Helmotz}, where $\bU^0\in\mathrm{X}_0$, $\bV^0\in\mathrm{X}_0$, $\nabla p\in\nabla\mathcal{S}$ and $\nabla q\in\nabla\mathcal{S}$. From the definition of $\mathrm{X}_0$, we have $\mathcal{M}(\bU^0,\nabla \phi)=0$ for any $\phi\in\mathcal{S}$. Thus,  we can further reduce the variational  form \eqref{eq:VF}.  Find $\bU\in H^1(\mathrm{curl},B_R\backslash\overline{D})$ such that $\mathcal{M}(\bU,\bV)=\mathcal{F}(\bV)$ for all $\bV\in H^1(\mathrm{curl},B_R\backslash\overline{D})$ with respect to the following problem: find $\bU\in H^1(\mathrm{curl},B_R\backslash\overline{D})$ such that
$$
\mathcal{M}(\nabla p,\nabla q)+\mathcal{M}(\nabla p,\bV^0)+\mathcal{M}(\bU^0,\bV^0)=\mathcal{F}(\nabla q)+\mathcal{F}(\bV^0),\quad  \forall\, q\in \mathcal{S},\,\, \forall\, \bU^0\in \mathrm{X}_0.
$$
From Lemma \ref{lem:compact1}, the solution $p\in\mathcal{S}$ can be uniquely determined from the identity $\mathcal{M}(\nabla p,\nabla q)=\mathcal{F}(\nabla q)$ for any $q$ in $\mathcal{S}$.Then the variation problem  \eqref{eq:VF} degenerates into finding $\bU^0\in \mathrm{X}_0$ by solving the equation,
\begin{equation*}
\mathcal{M}(\bU^0,\bV^0)=\widetilde{\mathcal{F}}(\bV^0),\quad\forall\, q\in \mathcal{S},
\end{equation*}
where $\widetilde{\mathcal{F}}(\bV^0):=\mathcal{F}(\bV^0)-\mathcal{M}(\nabla p,\nabla q)$ is a linear operator from $\mathrm{X}_0$ to $\mathbb{C}$.
To this end, we split the sesquilinear form $\mathcal{M}(\cdot,\cdot)$ into two parts
$$
\mathcal{M}(\bU^0,\bV^0)=\mathcal{M}_1(\bU^0,\bV^0)+\mathcal{M}_2(\bU^0,\bV^0),
$$
where
\begin{align*}
&\mathcal{M}_1(\bU^0,\bV^0):=(\mu_r^{-1}\nabla\times\bU^0,\nabla\times\bV^0)_{L^2(B_R\backslash\overline{D})^3}+(\varepsilon_r\bU^0,\bV^0)_{L^2(B_R\backslash\overline{D})^3}\nonumber\\
&\qquad\qquad\qquad\qquad\qquad\qquad\qquad\qquad\quad+\mathrm{i}k\,\big<G^2_e(\hat{\bx}\times\bU^0),\bV^0_{\mathrm{T}}\big>_{L^2(\partial B_R)^3},\\
&\mathcal{M}_2(\bU^0,\bV^0):=\mathrm{i}k\,\big<G^1_e(\hat{\bx}\times\bU^0),\bV^0_{\mathrm{T}}\big>_{L^2(\partial B_R)^3}-\Big((k^2+1)\varepsilon_r\,\bU_0,\bV_0\Big)_{L^2(B_R\backslash\overline{D})^3}.
\end{align*}
From the Cauchy-Schwarz inequality, we have
$$
|\mathcal{M}_1(\bU^0,\bV^0)|\leq C_1\|\bU^0\|_{H^1(\mathrm{curl}, \Omega\backslash\overline{D})}\|\bV^0\|_{H^1(\mathrm{curl}, \Omega\backslash\overline{D})}\quad \mbox{for}\,\, C_1>0.
$$
Taking the real and imaginary parts, and then using $(ii)$ of Lemma \ref{lem:proper1}, we have
$$
|\mathcal{M}_1(\bU^0,\bU^0)|\geq C_2\|\bU^0\|^2_{H^1(\mathrm{curl}, \Omega\backslash\overline{D})}\quad \mbox{for}\,\, C_2>0.
$$

By the Lax-Milgram lemma,  $\mathcal{M}_1$ is a bijective operator. From the lemma \ref{lem:imbed}, we see that the embedding $\mathrm{X}_0\rightarrow L^2(B_R\backslash\overline{D})^3$ is compact, which together with $(i)$ from the lemma \ref{lem:proper1}, which can easily be used to imply that $\mathcal{M}_2$ is a compact operator. Then, together with Theorem \ref{th:unique}, Fredholm alternative theorem and the definition of $\mathcal{F}$ given by \eqref{eq:F(v)}, we derive that there exists a unique solution $\bU$ of \eqref{eq:VF} associated with the estimate.
\begin{align*}
\|\bU\|_{H^1(\mathrm{curl},B_R\backslash\overline{D})}\leq C\big(\|\bff\|_{H^{-1/2}_{\mathrm{div}}(\partial \Omega)}+\|\bh\|_{H^{-1/2}_{\mathrm{div}}(\partial \Omega)}+\|\bq\|_{H^{-1/2}_{\mathrm{div}}(\partial D)}+\|\bJ\|_{L^2(B_R\backslash\overline{\Omega})^3}\big).
\end{align*}
Let $\bE=\bU|_{\Omega\backslash\overline{D}}$, $\bH=\frac{\nabla\times \bE}{\mathrm{i}k}$, $\bE^s=\bU-\bE_f$ in $\mathbb{R}^3\backslash\overline{\Omega}$  and $\bH^s=\frac{\nabla\times\bE^s}{\mathrm{i}k}$. We can easily check that $(\bE,\bH,\bE^s,\bH^s)$ is the unique solution of \eqref{MD:tran2} and admits the estimation  \eqref{ieq:est}.
\end{proof}

\subsection{Auxiliary lemmas for Case 2 of Theorem \ref{thm:main1}}
In this subsection, we also provide several lemmas for proving Case 2 in Theorem \ref{thm:main1} and consider the following scattering problems: Given $\bq\in H^{-1/2}_{\mathrm{div}}(\partial D)$, $\bff\in H^{-1/2}_{\mathrm{div}}(\partial \Omega)$, $\bh\in H^{-1/2}_{\mathrm{div}}(\partial \Omega)$, and $\bJ $ with $\supp(\bJ)\Subset B_{R_0}\backslash\overline{\Omega}\Subset B_{R}\backslash\overline{\Omega}$,
find $(\bE,\bH,\bE^s,\bH^s)\in H^1(\mathrm{curl}, \Omega\backslash\overline{D})\times H^1(\mathrm{curl}, \Omega\backslash\overline{D})\times  H^1_{\mathrm{loc}}(\mathrm{curl}, \mathbb{R}^3\backslash\overline{\Omega})\times H^1_{\mathrm{loc}}(\mathrm{curl}, \mathbb{R}^3\backslash\overline{\Omega})$ satisfying
\begin{align}\label{MD:tran2-2}
\left\{ \begin{array}{ll}
\nabla\times\bE-\mathrm{i}k\mu_r\bH={\bf0},\quad \nabla\times\bH+\mathrm{i}k\varepsilon_r\bE={\bf0}&\mbox{in}\ \  \Omega\backslash \overline{D},\quad \\[5pt]
\nabla\times\bE^s-\mathrm{i}k\bH^s={\bf0},\quad \nabla\times\bH^s+\mathrm{i}k\bE^s={\bJ}& \mbox{in}\ \ \mathbb{R}^3\backslash\overline{\Omega},\\[5pt]
\bnu\times\bE-\bnu\times\bE^s=\bff,\quad \bnu\times\bH-\bnu\times\bH^s=\bh&\mbox{on} \ \ \partial{\Omega},\\[5pt]
\bnu\times\bE=\bq&\mbox{on}\ \ \partial{D},\\[5pt]
\lim\limits_{|\bx| \to\infty}\big(  \bH^s\times \bx   -|\bx| \bE^s  \big)={\bf0}.
\end{array}
 \right.
\end{align}

Similar to the system \eqref{MD:tran2}, we give a variational formulation of \eqref{MD:tran2-2} over a bounded domain by using the exterior Calder\'{o}n operator $G_e$ on the artificial boundary $\partial B_R$. We consider the following space
$$W:=\Big\{\bw\in H^1(\mathrm{curl},B_R\backslash\overline{D})\Big|\bnu\times\bw={\bf 0}\,\,\mbox{on}\,\, \partial D\Big\}.$$
Similar to \eqref{eq:Ef}--\eqref{ieq:Ef}, for the fixed $\bq\in H^{-1/2}_{\mathrm{div}}(\partial D) $, we can easily find the unique solution $\bE_0$ for the PDE system
\begin{equation*}
\begin{cases}
 \nabla\times\nabla\times \bw-k^2\bw={\bf 0}&\mbox{in}\ \ B_R\backslash\overline{D},\\
 \nabla\cdot \bw=0&\mbox{in}\ \ B_R\backslash\overline{D},\\
 \bnu\times \bw=\bq &\mbox{on}\ \ \partial D,\\
 \bnu\times \bw={\bf 0} &\mbox{on}\ \ \partial B_R,
\end{cases}
\end{equation*}
with the following estimate
\begin{align}\label{def:E0}
\|\bE_0\|_{H^1(\mathrm{curl},B_R\backslash\overline{D})}\leq C\|\bq\|_{H^{-1/2}_{\mathrm{div}}(\partial D)}\quad\mbox{for}\quad C>0,
\end{align}
where $k^2$ is not an eigenvalue of the homogeneous Dirichlet problem with $\bnu\times\bw={\bf 0}$ on
the boundary $\partial D\cup\partial B_R$.

By eliminating the magnetic fields, using the transmission conditions in \eqref{MD:tran2-2} and the definition of $G_e$, along with integrating by parts, we obtain the corresponding variational formulation of \eqref{MD:tran2-2}:
Find $\bU\in H^1(\mathrm{curl},B_R\backslash\overline{D})$ satisfying
\begin{equation}\label{eq:VF2}
\mathcal{N}(\bU-\bE_0,\bPhi)=\mathcal{Q}(\bPhi) \quad\mbox{for all}\,\,\bPhi\in W,
\end{equation}
with
\allowdisplaybreaks
\begin{align}
&\mathcal{N}(\bU-\bE_0,\bPhi):=\int_{B_R\backslash\overline{D}}\mu_r^{-1}\big[\nabla\times(\bU-\bE_0)\big]\cdot(\nabla\times\overline{\bPhi})-k^2\varepsilon_r(\bU-\bE_0)\cdot\overline{\bPhi}\,\mathrm{d}\bx\nonumber\\
&\qquad\qquad\qquad+\mathrm{i}k\int_{\partial B_R}G_e\big[\hat{\bx}\times(\bU-\bE_0)\big]\cdot\overline{\bPhi}_{\mathrm{T}}\,\mathrm{d}\sigma\label{eq:A(U,V)2}\\
&\mathcal{Q}(\bPhi):=-\int_{\partial \Omega}\mathrm{i}k\bh\cdot\overline{\bPhi}_{\mathrm{T}}\,\mathrm{d}\sigma-\mathrm{i}k\int_{\partial B_R}G_e(\hat{\bx}\times\bE_{\bff})\cdot\overline{\bPhi}_{\mathrm{T}}\,\mathrm{d}\sigma+\int_{B_R\backslash\overline{\Omega}}\mathrm{i}k \bJ\cdot\overline{\bPhi}\,\mathrm{d}\bx\nonumber\\
&\qquad\qquad-\int_{B_R\backslash\overline{\Omega}}(\nabla\times \bE_{\bff})\cdot(\nabla\times\overline{\bPhi})-k^2\bE_{\bff}\cdot\overline{\bPhi}\,\mathrm{d}\bx-\int_{\partial B_R}G_e(\hat{\bx}\times\bE_0)\cdot\overline{\bPhi}_{\mathrm{T}}\,\mathrm{d}\sigma\nonumber\\
&\qquad\qquad-\int_{B_R\backslash\overline{D}}\mu^{-1}_r\nabla\times\bE_0\cdot\nabla\times\bPhi-k^2\varepsilon_r\bE_0\cdot\overline{\bPhi}\,\mathrm{d}\bx.\label{eq:F(v)2}
\end{align}
where $\bE_{\bff}$ is given by \eqref{eq:Ef} and satisfies \eqref{ieq:Ef}, $\bp\in H^{-1/2}_{\mathrm{div}}(\partial D)$, $\bff\in H^{-1/2}_{\mathrm{div}}(\partial \Omega)$, $\bh\in H^{-1/2}_{\mathrm{div}}(\partial \Omega)$ and $\bJ $ with $\supp(\bJ)\Subset B_{R_0}\backslash\overline{\Omega}\Subset B_{R}\backslash\overline{\Omega}$ are given.
If $\bU$ is a solution of \eqref{eq:VF2}, it is straightforward to show, using sufficiently smooth test functions $\bPhi$, that $\bE|_{D}:=\bU|_{D}$ and $\bE^s:=\bU|_{B_R\backslash\overline{\Omega}}+\bE_{\bff}$ satisfy the differential equations for the electric fields of \eqref{MD:tran2-2} in $\Omega\backslash\overline{D}$ and $B_R\backslash\overline{\Omega}$, respectively. Additionally, they satisfy the transmission conditions on $\partial \Omega$ and $\hat{\bx}\times(\nabla\times\bE^s)=\mathrm{i}k G_e(\hat{\bx}\times\bE^s)$ on $\partial B_R$. A solution of the variational problem \eqref{eq:VF} and the corresponding magnetic fields $\bH^s=\frac{\nabla\times \bE^s}{\mathrm{i}k}$ and $\bH=\frac{\nabla\times \bE}{\mathrm{i}k}$ can be extended to a solution of \eqref{MD:tran2-2}. Indeed, it follows directly that $\hat{\bx}\times \bE^s$ and $\hat{\bx}\times \bH^s$ are continuous across the interface $\partial B_R$.

\begin{thm}\label{th:unique2}
The problems \eqref{MD:tran2-2} and \eqref{eq:VF2} have at most one solution.
\end{thm}
\begin{proof}
The approach we adopt closely resembles the proof technique employed in establishing Theorem \ref{th:unique}, and to avoid redundancy, we omit its repetition here.
\end{proof}

Next, we will demonstrate the existence of a solution to problem \eqref{eq:VF2}, and we will also utilize the Helmholtz decomposition of the space $W$. To do this, we introduce the following spaces:
\begin{equation*}
\widetilde{\mathcal{S}}:=\Big\{\psi\in H^1(B_R\backslash\overline{D})\Big|\int_{\partial B_R}\psi\,\mathrm{d}\sigma=0\,\,\mbox{and}\,\, \bnu\times\nabla\psi={\bf0} \,\,\mbox{on}\,\,\partial D\Big\}
\end{equation*}
and
\begin{align*}
\widetilde{\mathrm{X}}_0:&=\Big\{ \widetilde{\bU}\in W\Big|-k^2\int_{\Omega\backslash\overline{D}} \varepsilon_r\widetilde{\bU}\cdot\nabla\overline{\psi}\,\mathrm{d}\bx-k^2\int_{B_R\backslash\overline{\Omega}}\widetilde{\bU}\cdot\nabla\overline{\psi}\,\mathrm{d}\bx\nonumber\\
&\qquad\qquad\qquad\qquad\qquad\qquad+\mathrm{i}k\int_{\partial B_R}G_e(\hat{\bx}\times\widetilde{\bU})\cdot\nabla_{\partial B_R}\overline{\psi}\,\mathrm{d}\sigma=0
\,\,\mbox{for all }\,\, \psi\in \mathcal{S}         \Big\}\nonumber\\
&=\Big\{\widetilde{\bU}\in W\Big|\nabla(\varepsilon_r\widetilde{\bU})=0\,\, \mbox{in}\,\, B_R\backslash\overline{D},\,\, -k^2\,\hat{\bx}\cdot\widetilde{\bU}=\mathrm{i}k\,\nabla_{\partial B_R}\cdot G_e(\hat{\bx}\times\widetilde{\bU})\,\, \nonumber\\
&\qquad\qquad\qquad\qquad\qquad\qquad\qquad\mbox{on}\,\,\partial B_R\,\,\mbox{and}\,\, \bnu\cdot\widetilde{\bU}=0\,\, \mbox{on}\,\, \partial D                                      \Big\},
\end{align*}
where $\varepsilon_r$ is given by \eqref{eq:U_E} and $\nabla_{\partial B_R}\cdot(\hat{\bx}\times\widetilde{\bU})$ is defined in \eqref{eq:grad_BR}. We can easily check that $\widetilde{\mathcal{S}}$ is a Hilbert space with the norm $\big\|\cdot\big\|_{H^1(B_R\backslash\overline{D})^3}$. Then, we introduce the following two sesquilinear forms $\widetilde{\mathcal{A}}_1(\cdot,\cdot)$ and $\widetilde{\mathcal{A}}_2(\cdot,\cdot)$ on $\widetilde{\mathcal{S}}\times\widetilde{\mathcal{S}}$ such that
$$
\mathcal{N}(\nabla\phi,\nabla\psi)=\widetilde{\mathcal{A}}_1(\nabla\phi,\nabla\psi)+\widetilde{\mathcal{A}}_2(\nabla\phi,\nabla\psi)\quad \mbox{for all}\,\, \phi,\psi\in \widetilde{\mathcal{S}},
$$
where
\begin{align*}
\widetilde{\mathcal{A}}_1(\nabla\phi,\nabla\psi):&=-k^2\int_{B_R\backslash\overline{D}}\varepsilon_r(\bx)\nabla\phi\cdot\nabla\overline{\psi}\,\mathrm{d}\bx+k^2\int_{\partial B_R}\widetilde{G}_e(\hat{\bx}\times\nabla\phi)\cdot\nabla_{\partial B_R}\overline{\psi}\mathrm{d}\sigma,\\
\widetilde{\mathcal{A}}_2(\nabla\phi,\nabla\psi):&=\mathrm{i}k\int_{\partial B_R}(G_e+\mathrm{i}\widetilde{G}_e)(\hat{\bx}\times\nabla\overline{\psi})\mathrm{d}\sigma.
\end{align*}


Similar to Lemmas \ref{lem:compact1} -Lemma \ref{lem:imbed}, there are additional lemmas that can be proven using a similar strategy. Therefore, we omit them.

\begin{lem}\label{lem:compact1-2}
The sesquilinear form $\widetilde{\mathcal{A}}_1(\cdot,\cdot):\widetilde{\mathcal{S}}\times\widetilde{\mathcal{S}}\rightarrow \mathbb{C}$ is bounded and coercive. Also, there exists a compact operator $K_2$ on $\widetilde{\mathcal{S}}$ which satisfies $\widetilde{\mathcal{A}}_2(\phi,\xi)=\widetilde{\mathcal{A}}_1(K_2\phi,\xi)$ for all $\phi$, $\xi\in \widetilde{\mathcal{S}}$. Furthermore, $I+K_2$ is an isomorphism from $\mathcal{S}$ onto itself.
\end{lem}

\begin{lem}\label{lem:Helmotz2}
The spaces $\nabla\widetilde{\mathcal{S}}$ and $\widetilde{\mathrm{X}}_0$ are closed subspaces of $W$ and $W=\nabla\widetilde{\mathcal{S}}\bigoplus\widetilde{\mathrm{X}}_0$. In addition, the projection operators from $W$ onto these two subspaces are both bounded, and there exist positive constants $C_1$ and $C_2$ satisfying
\begin{align*}
C_1\|\widetilde{\bU}+\nabla\psi\|_{H^1(\mathrm{curl},B_R\backslash\overline{D})}&\leq \|\widetilde{\bU}\|_{H^1(\mathrm{curl},B_R\backslash\overline{D})}+\|\nabla\psi\|_{H^1(\mathrm{curl},B_R\backslash\overline{D})}
\nonumber\\[1mm]
&\leq C_2\|\widetilde{\bU}+\nabla\psi\|_{H^1(\mathrm{curl},B_R\backslash\overline{D})}\quad \mbox{for all  $\widetilde{\bU}\in \widetilde{\mathrm{X}}_0$ and $\psi\in\widetilde{\mathcal{S}}$}.
\end{align*}
\end{lem}

\begin{lem}\label{lem:imbed2}
The embedding operator from $\widetilde{\mathrm{X}}_0$ into $ L^2(B_R\backslash\overline{D})^3$ is compact.
\end{lem}

The following result is based on Lemmas \ref{lem:compact1} to
\ref{lem:imbed} and Theorem \ref{th:unique2}.

\begin{thm}\label{th:exitence2}
The scattering  problem \eqref{MD:tran2-2} has a unique solution $(\bE,\bH,\bE^s,\bH^s)\in H^1(\mathrm{curl}, \Omega\backslash\overline{D})\times H^1(\mathrm{curl}, \Omega\backslash\overline{D})\times  H^1_{\mathrm{loc}}(\mathrm{curl}, \mathbb{R}^3\backslash\overline{\Omega})\times H^1_{\mathrm{loc}}(\mathrm{curl}, \mathbb{R}^3\backslash\overline{\Omega})$. Furthermore, one has the following estimate
\begin{align}\label{ieq:est2}
&\|\bE\|_{H^1(\mathrm{curl},\Omega\backslash\overline{D})}+\|\bH\|_{H^1(\mathrm{curl},\Omega\backslash\overline{D})}+\|\bE^s\|_{H^1(\mathrm{curl}, B_R\backslash\overline{\Omega})}+\|\bH^s\|_{H^1(\mathrm{curl}, B_R\backslash\overline{\Omega})}\nonumber\\
&\qquad\qquad\leq C\Big(\|\bff\|_{H^{-1/2}_{\mathrm{div}}(\partial \Omega)}+\|\bh\|_{H^{-1/2}_{\mathrm{div}}(\partial \Omega)}+\|\bq\|_{H^{-1/2}_{\mathrm{div}}(\partial D)}+\|\bJ\|_{L^2(B_R\backslash\overline{\Omega})^3}
\Big)
\end{align}
\end{thm}
\begin{proof}
 At first, we study the variation problem below: Find $\widetilde{\bU}\in W$ satisfying
\begin{equation}\label{eq:VF2v}
\mathcal{N}(\widetilde{\bU},\bPhi)=\mathcal{Q}(\bPhi) \quad\mbox{for all}\,\,\bPhi\in W,
\end{equation}
where these sesquilinear forms $\mathcal{N}(\cdot,\cdot)$ and $\mathcal{Q}(\cdot)$ are given by \eqref{eq:A(U,V)2} and \eqref{eq:F(v)2}, respectively. From Lemma \ref{lem:Helmotz2}, we know that for any $\widetilde{\bU}\in W$ and $\widetilde{\bV}\in W$, they have these decompositions $\widetilde{\bU}=\widetilde{\bU}^0+\nabla p$ and $\widetilde{\bV}=\widetilde{\bV}^0+\nabla q$, where $\widetilde{\bU}^0\in\widetilde{\mathrm{X}}_0$, $\widetilde{\bV}^0\in\widetilde{\mathrm{X}}_0$, $\nabla p\in\nabla\widetilde{\mathcal{S}}$ and $\nabla q\in\nabla\widetilde{\mathcal{S}}$. It is a direct result that $\mathcal{N}(\widetilde{\bU}^0,\nabla \phi)=0$ for any $\phi\in\widetilde{\mathcal{S}}$ from the definition of $\widetilde{\mathrm{X}}_0$. Next, we want to find $\widetilde{\bU}\in W$ such that
$$
\mathcal{N}(\nabla p,\nabla q)+\mathcal{N}(\nabla p,\widetilde{\bV}^0)+\mathcal{N}(\widetilde{\bU}^0,\widetilde{\bV}^0)=\mathcal{Q}(\nabla q)+\mathcal{Q}(\widetilde{\bV}^0),\quad  \forall\, q\in \widetilde{\mathcal{S}},\,\, \forall\,\, \widetilde{\bU}^0\in \widetilde{\mathrm{X}}_0.
$$

Using  the lemma \ref{lem:compact1-2}, we can uniquely determine the solution $p\in\mathcal{S}$ from the identity $\mathcal{N}(\nabla p,\nabla q)=\mathcal{Q}(\nabla q)$ for any $q\in \mathcal{S}$ and then reduce the above variation problem to finding $\widetilde{\bU}^0\in \widetilde{\mathrm{X}}_0$ by solving the equation
$$
\mathcal{N}(\widetilde{\bU}^0,\widetilde{\bV}^0)=\mathcal{Q}(\widetilde{\bV}^0)-\mathcal{N}(\nabla p,\nabla q),\quad\forall\, q\in \widetilde{\mathcal{S}}.
$$
In what follows, the sesquilinear form $\mathcal{N}_1(\cdot,\cdot)$ and $\mathcal{N}_2(\cdot,\cdot)$ are given by
\begin{align*}
&\mathcal{M}_1(\widetilde{\bU}^0,\widetilde{\bV}^0):=(\mu_r^{-1}\nabla\times\widetilde{\bU}^0,\nabla\times\widetilde{\bV}^0)_{L^2(B_R\backslash\overline{D})^3}+(\varepsilon_r\widetilde{\bU}^0,\widetilde{\bV}^0)_{L^2(B_R\backslash\overline{D})^3}\nonumber\\
&\qquad\qquad\qquad\qquad\qquad\qquad\qquad\qquad\quad+\mathrm{i}k\,\big<\widetilde{G}^2_e(\hat{\bx}\times\widetilde{\bU}^0),\widetilde{\bV}^0_{\mathrm{T}}\big>_{L^2(\partial B_R)^3},\\[2pt]
&\mathcal{M}_2(\widetilde{\bU}^0,\widetilde{\bV}^0):=\mathrm{i}k\,\big<\widetilde{G}^1_e(\hat{\bx}\times\widetilde{\bU}^0),\widetilde{\bV}^0_{\mathrm{T}}\big>_{L^2(\partial B_R)^3}-\big((k^2+1)\varepsilon_r\,\widetilde{\bU}^0,\widetilde{\bV}^0\big)_{L^2(B_R\backslash\overline{D})^3},
\end{align*}
satisfying
$$
\mathcal{N}(\widetilde{\bU}^0,\widetilde{\bV}^0)=\mathcal{N}_1(\widetilde{\bU}^0,\widetilde{\bV}^0)+\mathcal{N}_2(\widetilde{\bU}^0,\widetilde{\bV}^0)=\widetilde{\mathcal{Q}}(\bV_0),
$$
where $\widetilde{\mathcal{Q}}(\bV_0):=\mathcal{Q}(\widetilde{\bV}^0)-\mathcal{N}(\nabla p,\nabla q)$, and
$\widetilde{G}^1_e$ and $\widetilde{G}^2_e$ are defined as $G^1_e$ and $G^2_e$ in Lemma \ref{lem:proper1}.
From the Cauchy-Schwarz inequality, we have that
$$
|\mathcal{N}_1(\widetilde{\bU}^0,\widetilde{\bV}^0)|\leq C_1\|\widetilde{\bU}^0\|_{H^1(\mathrm{curl}, \Omega\backslash\overline{D})}\|\widetilde{\bV}^0\|_{H^1(\mathrm{curl}, \Omega\backslash\overline{D})}\quad \mbox{for}\,\, \widetilde {C}_1>0.
$$
Taking the real and imaginary parts, and then using the lemma  \ref{lem:compact1-2},  we get that
$$
|\mathcal{N}_1(\widetilde{\bU}^0,\widetilde{\bV}^0)|\geq \widetilde{C}_2\|\widetilde{\bU}^0\|^2_{H^1(\mathrm{curl}, \Omega\backslash\overline{D})}\quad \mbox{for}\,\, \widetilde{C}_2>0.
$$

Owing to the Lax-Milgram Lemma, $\mathcal{N}_1$ is a bijective operator. From Lemma \ref{lem:imbed2}, we see that the embedding $\widetilde{\mathrm{X}}_0\rightarrow L^2(B_R\backslash\overline{D})^3$ is compact. Together with Lemma \ref{lem:proper1}, this implies that $\mathcal{N}_2$ is a compact operator. The uniqueness of $\bE_0\in H^1(\mathrm{curl},B_R\backslash\overline{D})$ satisfying \eqref{def:E0}, and based on Theorem \ref{th:unique2}, it is clear that there is at most one solution to problem \eqref{eq:VF2v}. By employing the Fredholm alternative theorem, we deduce that the problem \eqref{eq:VF2v} possesses a unique solution $\widetilde{\bU}$. Consequently, the problem \eqref{eq:VF2} also admits a unique solution $\bU=\widetilde{\bU}+\bE_0$, with the following estimate:
 \begin{align*}
\|\bU\|_{H^1(\mathrm{curl},B_R\backslash\overline{D})}&\leq \|\widetilde{\bU}\|_{H^1(\mathrm{curl},B_R\backslash\overline{D})} +\|\bE_0\|_{H^1(\mathrm{curl},B_R\backslash\overline{D})}\nonumber\\
&\leq C\big(\|\bff\|_{H^{-1/2}_{\mathrm{div}}(\partial \Omega)}+\|\bh\|_{H^{-1/2}_{\mathrm{div}}(\partial \Omega)}+\|\bq\|_{H^{-1/2}_{\mathrm{div}}(\partial D)}+\|\bJ\|_{L^2(B_R\backslash\overline{\Omega})^3}\big).
 \end{align*}
At last, let $\bE=\big(\widetilde{\bU}+\bE_0\big)\big|_{\Omega\backslash\overline{D}}$, $\bH=\frac{\nabla\times \bE}{\mathrm{i}k}$, $\bE^s=\widetilde{\bU}+\bE_0+\bE_{\bff}$ in $\mathbb{R}^3\backslash\overline{\Omega}$  and $\bH^s=\frac{\nabla\times\bE^s}{\mathrm{i}k}$. It can be easily proved that $(\bE,\bH,\bE^s,\bH^s)$ is the unique solution of \eqref{MD:tran2-2} and admits the estimate \eqref{ieq:est2}.
\end{proof}

\section{Results on the effective medium scattering problems}\label{sect:medium}


This section aims to identify a medium scattering problem related to the obstacle scattering problem \eqref{eq:sca1}, both associated with the same incident wave denoted as $\bE^i$. The corresponding electromagnetic medium $(\widetilde{\Omega}; \widetilde{\mu}_r, \widetilde{\varepsilon}_r)$ satisfies the restriction: $(\widetilde{\mu}_r,\widetilde{\varepsilon}_r)\big|_{\Omega\backslash\overline{D}}=(\mu_r,\varepsilon_r)$, and $(D; \mu_D, \varepsilon_D):=(\widetilde{\Omega}; \widetilde{\mu}_r, \widetilde{\varepsilon}_r)\big|_{D}$ is considered as a $\delta^{1/2}$-realization of the obstacle $D$ in the sense of Definition \ref{def:2}. We shall separately prove that the parameters $\mu_D$ and $\varepsilon_D$ chosen as per Theorem \ref{thm:main1} adequately in the obstacle $D$.

\subsection{Results for Case 1 of Theorem \ref{thm:main1}}
In this subsection, we mainly consider  the medium scattering system as follows:
\begin{align}\label{ME:tran1}
\left\{ \begin{array}{ll}
\nabla\times\bE_\delta-\mathrm{i}k\mu_D\bH_\delta={\bf0},\qquad \nabla\times\bH_\delta+\mathrm{i}k\varepsilon_D\bE_\delta={\bf0}&\mbox{in}\ \  D,\quad \\[5pt]
\nabla\times\bE_\delta-\mathrm{i}k\mu_r\bH_\delta={\bf0},\qquad \nabla\times\bH_\delta+\mathrm{i}k\varepsilon_r\bE_\delta={\bf0}&\mbox{in}\ \  \Omega\backslash \overline{D},\quad \\[5pt]
\nabla\times\bE_\delta^s-\mathrm{i}k\bH_\delta^s={\bf0},\qquad \ \ \nabla\times\bH_\delta^s+\mathrm{i}k\bE_\delta^s={\bJ}& \mbox{in}\ \ \mathbb{R}^3\backslash\overline{\Omega},\\[5pt]
\bnu\times\bE_\delta\big|^-_{\partial D}=\bnu\times\bE_\delta\big|^+_{\partial D},\quad \bnu\times\bH_\delta\big|^-_{\partial D}=\bnu\times\bH_\delta\big|^+_{\partial D}&\mbox{on} \ \ \partial D,\\[5pt]
\bnu\times\bE_\delta=\bnu\times(\bE_\delta^s+\bE^i),\quad \bnu\times\bH_\delta=\bnu\times(\bH_\delta^s+\bH^i)&\mbox{on} \ \ \partial{\Omega},\\[5pt]
\lim\limits_{|\bx| \to\infty}\big(  \bH_\delta^s\times \bx   -|\bx| \bE_\delta^s  \big)={\bf0},
\end{array}
 \right.
\end{align}
where $\mu_D$ and $\varepsilon_D$ are given by \eqref{eq:eff2}. The notation  \lq\lq$\cdot\big|^{\pm}{\partial D}$\rq\rq denotes the limits of certain fields from outside and inside $\partial D$, respectively. Similarly, the notation \lq\lq$\cdot\big|^{\pm}{\partial \Omega}$\rq\rq is defined analogously. By eliminating the magnetic fields and combining this with the exterior Calder\'{o}n operator $G_e$ given in \eqref{eq:Calderon}, the system \eqref{ME:tran1} is transformed into\begin{align}\label{ME:tran1_E}
\left\{ \begin{array}{ll}
\nabla\times(\mu_D^{-1}\nabla\times\bE_\delta)-k^2\varepsilon_D\bE_\delta={\bf0}&\mbox{in}\ \  D,\quad \\[5pt]
\nabla\times(\mu_r^{-1}\nabla\times\bE_\delta)-k^2\varepsilon_r\bE_\delta={\bf0},&\mbox{in}\ \  \Omega\backslash \overline{D},\quad \\[5pt]
\nabla\times(\nabla\times\bE_\delta^s)-k^2\bE_\delta^s=\mathrm{i}k\bJ & \mbox{in}\ \ B_R\backslash\overline{\Omega},\\[5pt]
\bnu\times\bE_\delta\big|^-_{\partial D}=\bnu\times\bE_\delta\big|^+_{\partial D}&\mbox{on} \ \ \partial D,\\[5pt]
\bnu\times(\mu_D^{-1}\nabla\times\bE_\delta)\big|^-_{\partial D}=\bnu\times(\mu_r^{-1}\nabla\times\bE_\delta)\big|^+_{\partial D}&\mbox{on} \ \ \partial D,\\[5pt]
\bnu\times\bE_\delta=\bnu\times(\bE_\delta^s+\bE^i)&\mbox{on} \ \ \partial{\Omega},\\[5pt] \bnu\times(\mu_r^{-1}\nabla\times\bE_\delta)\big|^-_{\partial \Omega}=\bnu\times\big(\nabla\times(\bE_\delta^s+\bE^i)\big)&\mbox{on} \ \ \partial{\Omega},\\[5pt]
\hat{\bx}\times(\nabla\times \bE^s)=\mathrm{i}k G_e(\hat{\bx}\times \bE^s)&\mbox{on} \ \ \partial{B_R}.
\end{array}
 \right.
\end{align}
\allowdisplaybreaks
The lemma below asserts that the unique solution $\bE_\delta$ of the medium scattering problem \eqref{ME:tran1_E}, restricted to the domains $D$ and $\Omega\backslash\overline{D}$, can be bounded by $\bE^{i}$ and the source $\bJ$. This lemma plays a crucial role in the proof of Theorem \ref{thm:main1}.
\begin{lem}\label{lem:mediu1}
Let $\bE_\delta$ be the solution of \eqref{ME:tran1_E} for Case 1. Then there are positive constants $C_1$ and $C_2$ such that the following estimate  holds for all $\delta\ll 1$ and sufficiently large $R$:
\begin{align}\label{inq:esti1_m}
\sqrt{\delta}\|\bE_\delta\|_{H^1(\mathrm{curl},D)}\leq C_1\big( \|\nabla\times\bE^i\|_{H^1(\mathrm{curl},B_R\backslash\overline{\Omega})}+\|\bE^i\|_{H^1(\mathrm{curl},B_R\backslash\overline{\Omega})} + \|\bJ\|_{L^2(B_R\backslash\overline{D})^3} \big)
\end{align}
and
\begin{align}\label{inq:esti2_m}
\|\bE_\delta\|_{H^1(\mathrm{curl},B_R\backslash\overline{D})}\leq C_2\big(\|\nabla\times\bE^i\|_{H^1(\mathrm{curl},B_R\backslash\overline{\Omega})} +\|\bE^i\|_{H^1(\mathrm{curl},B_R\backslash\overline{\Omega})} + \|\bJ\|_{L^2(B_R\backslash\overline{D})^3} \big).
\end{align}
\end{lem}
\begin{proof}
Multiplying  $\nabla\times(\widetilde{\mu}_r\,\nabla\times \,\bE_\delta)-k^2\widetilde{\varepsilon}_r\,\bE_\delta=\mathrm{i}k\bJ$  by $\overline{\bE}_\delta$ and integrating it over $D$ and $\Omega\backslash \overline{D}$, respectively, we get
\begin{align*}
&\int_{D}\mu_D^{-1}|\nabla\times\bE_\delta|^2\mathrm{d}\bx+\int_{\partial D}\delta\bnu\times( \nabla\times \bE_\delta)\cdot\overline{\bE}_\delta\mathrm{d}\sigma-\int_{D}k^2\varepsilon_D|\bE_\delta|^2\mathrm{d}\bx=0,\\[2mm]
&\int_{\Omega\backslash\overline{D}}\mu^{-1}_r|\nabla\times\bE_\delta|^2\mathrm{d}\bx+\int_{\partial \Omega}\bnu\times (\nabla\times\bE_\delta)\big|^-_{\partial \Omega}\cdot(\mu^{-1}_r \overline{\bE}_\delta) \mathrm{d}\sigma-\int_{\Omega\backslash\overline{D}}k^2\varepsilon_r\bE_\delta\cdot\overline{\bE}_\delta\mathrm{d}\bx\\
&\quad\qquad\qquad\qquad\qquad\qquad\qquad-\int_{\partial D}\bnu\times (\nabla\times\bE_\delta)\big|^+_{\partial D}\cdot(\mu^{-1}_r \overline{\bE}_\delta) \mathrm{d}\sigma=0.
\end{align*}
By repeating the proceudures described above, replacing $\overline{\bE}_\delta$ with $\overline{\bE}^s_\delta$, and considering the integral domain as $B_R\backslash\overline{\Omega}$, one can deduce that
\begin{align*}
&\int_{B_R\backslash\overline{\Omega}}|\nabla\times\bE^s_\delta|^2\mathrm{d}\bx+\int_{\partial B_R}\bnu\times \nabla\times\bE^s_\delta\big|_{\partial \Omega}\cdot \overline{\bE}^s_\delta \mathrm{d}\sigma-\int_{B_R\backslash\overline{\Omega}}k^2|\bE^s_\delta|^2\mathrm{d}\bx\\
&\qquad\qquad\qquad\qquad\qquad\qquad-\int_{\partial \Omega}\bnu\times (\nabla\times\bE^s_\delta)\big|^+_{\partial D}\cdot \overline{\bE}^s_\delta \mathrm{d}\sigma=\int_{B_R\backslash\overline{\Omega}}\mathrm{i}k\bJ\cdot\overline{\bE}^s_\delta\mathrm{d}\bx.
\end{align*}
By adding up these integral identities, we have
{\small\begin{align*}
&\int_{D}\mu_D^{-1}|\nabla\times\bE_\delta|^2\mathrm{d}\bx+\int_{\partial D}\mu_D^{-1}\,\bnu\times (\nabla\times \bE_\delta)\cdot\overline{\bE}_\delta\mathrm{d}\sigma-\int_{D}k^2\varepsilon_D\bE_\delta\cdot\overline{\bE}_\delta\mathrm{d}\bx\,+\int_{\Omega\backslash\overline{D}}\mu^{-1}_r|\nabla\times\bE_\delta|^2\mathrm{d}\bx\nonumber\\
&+\int_{\partial \Omega}\bnu\times (\nabla\times\bE_\delta)\big|^-_{\partial \Omega}\cdot(\mu^{-1}_r \overline{\bE}_\delta) \mathrm{d}\sigma-\int_{\Omega\backslash\overline{D}}k^2\varepsilon_r\bE_\delta\cdot\overline{\bE}_\delta\mathrm{d}\bx-\int_{\partial D}\bnu\times (\nabla\times\bE_\delta)\big|^+_{\partial D}\cdot(\mu^{-1}_r \overline{\bE}_\delta) \mathrm{d}\sigma\nonumber\\
&+\int_{B_R\backslash\overline{\Omega}}|\nabla\times\bE^s_\delta|^2\mathrm{d}\bx+\int_{\partial B_R}\bnu\times (\nabla\times\bE^s_\delta)\big|_{\partial \Omega}\cdot \overline{\bE}^s_\delta \mathrm{d}\sigma-\int_{B_R\backslash\overline{\Omega}}k^2|\bE^s_\delta|^2\mathrm{d}\bx\nonumber\\
&-\int_{\partial \Omega}\bnu\times (\nabla\times\bE^s_\delta)\big|^+_{\partial D}\cdot \overline{\bE}^s_\delta \mathrm{d}\sigma=\int_{B_R\backslash\overline{\Omega}}\mathrm{i}k\bJ\cdot\overline{\bE}^s_\delta\rmd\bx.
\end{align*}}
By using the transmission conditions on $\partial D$ and $\partial \Omega$, and the definition of $G_e$, we obtain
{\small\begin{align}\label{eq:inter3 2'}
&\mu_D^{-1}\,\int_{D}|\nabla\times\bE_\delta|^2\mathrm{d}\bx-k^2\varepsilon_D\int_{D}|\bE_\delta|^2\mathrm{d}\bx\,+\int_{\Omega\backslash\overline{D}}\mu^{-1}_r|\nabla\times\bE_\delta|^2\mathrm{d}\bx-k^2\int_{\Omega\backslash\overline{D}}\varepsilon_r\bE_\delta\cdot\overline{\bE}_\delta\mathrm{d}\bx\nonumber\\
&+\int_{B_R\backslash\overline{\Omega}}|\nabla\times\bE^s_\delta|^2\mathrm{d}\bx-\int_{B_R\backslash\overline{\Omega}}k^2|\bE^s_\delta|^2\mathrm{d}\bx
+\int_{\partial \Omega}\gamma_t(\nabla\times\bE^s_\delta )\cdot(\overline{\bE}^i)_{\mathrm{T}} \rmd\sigma+\int_{\partial \Omega}\gamma_t(\nabla\times\bE^i )\cdot(\overline{\bE}^i)_{\mathrm{T}} \rmd\sigma\nonumber\\
&+\int_{\partial \Omega}\gamma_t(\nabla\times\bE^i )\cdot(\overline{\bE}_\delta^s)_{\mathrm{T}} \rmd\sigma
+\int_{\partial B_R}\gamma_t(\nabla\times\bE^s_\delta )\cdot(\overline{\bE}_\delta^s)_{\mathrm{T}}  \rmd\sigma=\int_{B_R\backslash\overline{\Omega}}\mathrm{i}k\bJ\cdot\overline{\bE}^s_\delta\rmd\bx.
\end{align}}

From $(\widetilde{\mu}_r, \widetilde{\varepsilon}_r)\big|_{\Omega\backslash\overline{D}}=(\mu_r, \varepsilon_r)\big|_{\Omega\backslash\overline{D}}$, $(\widetilde{\mu}_r, \widetilde{\varepsilon}_r)\big|_{D}=(\mu_D, \varepsilon_D)$, and the expressions of $\mu_D$ and $\varepsilon_D$ in \eqref{eq:eff2},
Taking the real and imaginary parts of \eqref{eq:inter3 2'}, it is easy to obtain
\begin{align}
&\delta\,\|\nabla\times\bE_\delta\|^2_{L^2(D)^3}-k^2\eta_0\|\bE_\delta\|^2_{L^2(D)^3}+\int_{\Omega\backslash\overline{D}}\mu^{-1}_r|\nabla\times\bE_\delta|^2\mathrm{d}\bx-k^2\int_{\Omega\backslash\overline{D}}\Re\varepsilon_r\,\bE_\delta\cdot\overline{\bE}_\delta\mathrm{d}\bx\nonumber\\
&\,\,+\|\nabla\times\bE^s_\delta\|^2_{L^2(B_R\backslash\overline{\Omega})^3}-k^2\|\bE^s_\delta\|^2_{L^2(B_R\backslash\overline{\Omega})^3}
+\Re\int_{\partial \Omega}\gamma_t(\nabla\times\bE^s_\delta )\cdot(\overline{\bE}^i)_{\mathrm{T}} \rmd\sigma\nonumber\\
&\,\,+\Re\int_{\partial \Omega}\gamma_t(\nabla\times\bE^i )\cdot(\overline{\bE}^i)_{\mathrm{T}} \rmd\sigma+\Re\int_{\partial \Omega}\gamma_t(\nabla\times\bE^i )\cdot(\overline{\bE}_\delta^s)_{\mathrm{T}}\rmd\sigma\nonumber\\
&\,\,+\Re\int_{\partial B_R}\gamma_t(\nabla\times\bE^s_\delta )\cdot(\overline{\bE}_\delta^s)_{\mathrm{T}}  \rmd\sigma=k\Im\int_{B_R\backslash\overline{\Omega}}\bJ\cdot\overline{\bE}^s_\delta\rmd\bx.\label{eq:inter4}
\end{align}
and
\begin{align}
&-k^2\tau_0\|\bE_\delta\|^2_{L^2(D)^3}-k^2\int_{\Omega\backslash\overline{D}}\Im\varepsilon_r\,\bE_\delta\cdot\overline{\bE}_\delta\mathrm{d}\bx+\Im\int_{\partial \Omega}\gamma_t(\nabla\times\bE^s_\delta )\cdot(\overline{\bE}^i)_{\mathrm{T}} \rmd\sigma\nonumber\\
&\,\,+\Im\int_{\partial \Omega}\gamma_t(\nabla\times\bE^i )\cdot(\overline{\bE}^i)_{\mathrm{T}} \rmd\sigma+\Im\int_{\partial \Omega}\gamma_t(\nabla\times\bE^i )\cdot(\overline{\bE}_\delta^s)_{\mathrm{T}}\rmd\sigma\nonumber\\
&\,\,+\Im\int_{\partial B_R}\gamma_t(\nabla\times\bE^s_\delta )\cdot(\overline{\bE}_\delta^s)_{\mathrm{T}}  \rmd\sigma=k\Re\int_{B_R\backslash\overline{\Omega}}\bJ\cdot\overline{\bE}^s_\delta\rmd\bx. \label{eq:inter6 2}
\end{align}
The following inequalities can be derived from \eqref{eq:inter4}--\eqref{eq:inter6 2}, along with the properties of $\gamma_t$ and $\gamma_T$ provided in Theorem \ref{th_auxility1}, and the trace theorem,
\begin{align*}
 \delta\|\nabla\times\bE_\delta\|^2_{L^2(D)^3}&\leq C_1\,\Big(\|\bE_\delta\|^2_{L^2(D)^3}+\|\nabla\times\bE_\delta\|^2_{L^2(\Omega\backslash\overline{D})^3}
 +\|\bE_\delta\|_{L^2(\Omega\backslash\overline{D})^3}\nonumber\\
 &\,\,+\|\nabla\times\bE^s_\delta\|^2_{L^2(B_R\backslash\overline{\Omega})^3}+\|\nabla\times\bE^s_\delta\|_{H^1(\mathrm{curl},B_R\backslash\overline{\Omega})} \|\bE^i\|_{H^1(\mathrm{curl},B_R\backslash\overline{\Omega})}  \nonumber\\
 &\,\,+\|\bE_\delta\|^2_{L^2(B_R\backslash\overline{\Omega})^3} +   \|\nabla\times\bE^i\|_{H^1(\mathrm{curl},B_R\backslash\overline{\Omega})}  \|\bE^i\|_{H^1(\mathrm{curl},B_R\backslash\overline{\Omega})} \nonumber\\
 &\,\,+\|\nabla\times\bE^i\|_{H^1(\mathrm{curl},B_R\backslash\overline{\Omega})}  \|\bE^s_\delta\|_{H^1(\mathrm{curl},B_R\backslash\overline{\Omega})} +\|\bJ\|_{L^2(B_R\backslash\overline{\Omega})^3}\|\bE^s_\delta\|_{L^2(B_R\backslash\overline{\Omega})^3}\nonumber\\
 &\,\,+\|\nabla\times\bE^s_\delta\|_{H^1(\mathrm{curl},B_R\backslash\overline{\Omega})} \|\bE^s_\delta\|_{H^1(\mathrm{curl},B_R\backslash\overline{\Omega})}
 \Big)\nonumber\\
&\leq C_2\,\big(\|\bE_\delta\|^2_{L^2(D)^3}+\|\bE_\delta\|^2_{H^1(\mathrm{curl},B_R\backslash\overline{D})}+\|\nabla\times\bE^i\|^2_{H^1(\mathrm{curl},B_R\backslash\overline{D})}\nonumber\\
&\qquad\quad+\|\bE^i\|^2_{H^1(\mathrm{curl},B_R\backslash\overline{\Omega})}+\|\bJ\|^2_{L^2(B_R\backslash\overline{D})^3}\Big)
\end{align*}
and
 \begin{align*}
\|\bE_\delta\|^2_{L^2(D)^3}\leq& C_3\,\Big(  \|\bJ\|_{L^2(B_R\backslash\overline{\Omega})^3}\|\bE^s_\delta\|_{L^2(B_R\backslash\overline{\Omega})^3}+ \|\nabla\times\bE^s_\delta\|_{H^1(\mathrm{curl},B_R\backslash\overline{\Omega})} \|\bE^i\|_{H^1(\mathrm{curl},B_R\backslash\overline{\Omega})} \nonumber\\
&+\|\nabla\times\bE^s_\delta\|_{H^1(\mathrm{curl},B_R\backslash\overline{\Omega})} \|\bE^s_\delta\|_{H^1(\mathrm{curl},B_R\backslash\overline{\Omega})}+\|\bE_\delta\|^2_{L^2(\Omega\backslash\overline{D})^3}  \nonumber\\
&+\|\nabla\times\bE^i\|_{H^1(\mathrm{curl},B_R\backslash\overline{\Omega})} \|\bE^i\|_{H^1(\mathrm{curl},B_R\backslash\overline{\Omega})}\nonumber\\
&+\|\nabla\times\bE^i\|_{H^1(\mathrm{curl},B_R\backslash\overline{\Omega})} \|\bE^s_\delta\|_{H^1(\mathrm{curl},B_R\backslash\overline{\Omega})} \Big)\nonumber\\
 \leq& C_4\,\big( \|\bE_\delta\|^2_{L^2(\Omega\backslash\overline{D})^3} +\|\bE^i\|^2_{H^1(\mathrm{curl},B_R\backslash\overline{\Omega})}+ \|\nabla\times\bE^i\|^2_{H^1(\mathrm{curl},B_R\backslash\overline{\Omega})}\nonumber\\
  &\qquad+ \|\bJ\|^2_{L^2(B_R\backslash\overline{\Omega})^3}\big),
\end{align*}
where $C_1$, $C_2$, $C_3$, and $C_4$ are positive constants depending on $q,\,\eta_0,\tau_0, \,B_R,\, k,\,\Omega,\,\mu_r$ and $D$.
Thus, we get
\begin{align}\label{eq:inter8 2}
\delta\,\|\bE_\delta\|^2_{H^1(\mathrm{curl},D)} \leq &C_5 \,\big( \|\bE_\delta\|^2_{H^1(\mathrm{curl},B_R\backslash\overline{D})}  +\|\bE^i\|^2_{H^1(\mathrm{curl},B_R\backslash\overline{\Omega})}+ \|\nabla\times\bE^i\|^2_{H^1(\mathrm{curl},B_R\backslash\overline{\Omega})} \nonumber\\
&\qquad+\|\bJ\|^2_{L^2(B_R\backslash\overline{\Omega})^3}\big),
\end{align}
where $C_5$ is a positive constant not relying on $\delta$.
Next, we aim to prove \eqref{inq:esti2_m} by contradiction. Suppose that for any nonnegative integer $n$, there exists a set of data ${(\bE_n, \bE^i_n,\bJ_n)}$, where $\bE_n$ represents the unique solution of \eqref{ME:tran1_E} with the input data $\bJ_n$ and $\bE^i_n$, satisfying
\begin{align*}
\left\{ \begin{array}{ll}
\|\bJ_n\|_{L^2(B_R\backslash\overline{\Omega})^3}+\|\nabla\times\bE^i_n\|_{H^1(\mathrm{curl},B_R\backslash\overline{\Omega})^3}+\|\bE^i_n\|_{H^1(\mathrm{curl},B_R\backslash\overline{\Omega})^3}=1	,\quad \\[5pt]
\|\bE_n\|_{H^1(\mathrm{curl},B_R\backslash\overline{D})}\rightarrow\infty \quad \mbox{as}\quad \delta\to 0.
\end{array}
 \right.
\end{align*}
Let $\left\{(\hat{\bE}_n, \hat{\bE}^i_n,\hat{\bJ}_n)\right\}$ be given by
\begin{align*}
\left\{ \begin{array}{ll}
\hat{\bJ}_n=\dfrac{\bJ_n}{\|\bE_n\|_{H^1(\mathrm{curl},B_R\backslash\overline{D})}},\quad \hat{\bE}_n=\dfrac{\bE_n}{\|\bE_n\|_{H^1(\mathrm{curl},B_R\backslash\overline{D})}},\\[15pt]
\hat{\bE}^i_n=\dfrac{\bE^i_n}{\|\bE_n\|_{H^1(\mathrm{curl},B_R\backslash\overline{D})}},\quad \hat{\bE}^s_n=\dfrac{\bE^s_n}{\|\bE_n\|_{H^1(\mathrm{curl},B_R\backslash\overline{D})}},\\[15pt]
\bE^s_n= \bE_n-\bE^i_n,\qquad \qquad\quad\,\,\hat{\bE}^s_n= \hat{\bE}_n-\hat{\bE}^i_n.
\end{array}
 \right.
\end{align*}
It is easy to check that $\hat{\bE}_n$ solves the system \eqref{ME:tran1_E} associated  with the input data $\hat{\bE}^i_n$ and $\hat{\bJ}_n$. When $\delta$ tends to $0$, one has
\begin{align}\label{data:1 2}
 \|\hat{\bJ}_n\|_{L^2(B_R\backslash\overline{D})^3},\|\hat{\bE}^i_n\|_{H^1(\mathrm{curl},B_R\backslash\overline{D})},\|\nabla\times\hat{\bE}^i_n\|_{H^1(\mathrm{curl},B_R\backslash\overline{D})}\,\to 0 \,\,\mbox{and}\,\,\|\hat{\bE}_n\|_{H^1(\mathrm{curl},B_R\backslash\overline{D})}=1.
\end{align}
Thus, we obtain the subsequent estimation by repeating similar procedures for$(\bE_\delta, \bE^i,\bJ)$,
\begin{align*}
\delta\,\|\hat{\bE}_n\|^2_{H^1(\mathrm{curl},D)} \leq &C_6 \,\big( \|\hat{\bE}_n\|^2_{H^1(\mathrm{curl},B_R\backslash\overline{D})}  +\|\hat{\bE}^i_n\|^2_{H^1(\mathrm{curl},B_R\backslash\overline{\Omega})} +\|\nabla\times\hat{\bE}^i_n\|^2_{H^1(\mathrm{curl},B_R\backslash\overline{\Omega})}\nonumber\\
&+ \|\hat{\bJ}_n\|^2_{L^2(B_R\backslash\overline{\Omega})^3}\big),
\end{align*}
where $C_6$ is a positive constant not relying on $\delta$. Combining this with \eqref{data:1 2}, we imply
\begin{equation}\label{ieq:09}
\sqrt{\delta}\|\hat{\bE}_n\|_{H^1(\mathrm{curl},D)}\leq C_7\quad\mbox{for some positive constant $C_7$}.
\end{equation}
We note that $\hat{\bH}_n=\frac{\hat{\bE}_n}{\mathrm{i}k\mu_r}$ and $\hat{\bH}^s_n=\frac{\hat{\bE}^s_n}{\mathrm{i}k}$. Hence,  $(\hat{\bE}_n\big|_{\Omega\backslash\overline{D}},\,\hat{\bH}_n\big|_{\Omega\backslash\overline{D}},\,\hat{\bE}^s_n\big|_{\mathbb{R}^3\backslash\overline{\Omega}},\,\hat{\bH}^s_n\big|_{\mathbb{R}^3\backslash\overline{\Omega}})$ is the unique solution of \eqref{MD:tran2} with these boundary conditions $\bq=\bnu\times(\mu_r^{-1}\nabla\times \hat{\bE}_n)\Big|^+_{\partial{D}}$, $\bff=\bnu\times\hat{\bE}^i_n\big|^+_{\partial \Omega}$ and $\mathrm{i}k \bh=\bnu\times(\nabla\times\hat{\bE}^i_n)\big|^+_{\partial \Omega}$. From Theorem \ref{th:exitence1} and (\ref{ieq:09}), we have
 \begin{align*}
\|\bE\|_{H^1(\mathrm{curl},B_R\backslash\overline{D})}&\leq C_8\big(\|\bnu\times\hat{\bE}^i_n\big|^+_{\partial \Omega}\|_{H^{-1/2}_{\mathrm{div}}(\partial \Omega)}+\|\bnu\times(\nabla\times\hat{\bE}^i_n)\big|^+_{\partial \Omega}\|_{H^{-1/2}_{\mathrm{div}}(\partial \Omega)}\\
&\qquad+\|\bnu\times(\mu_r^{-1}\nabla\times \hat{\bE}_n)\big|^+_{\partial{D}}\|_{H^{-1/2}_{\mathrm{div}}(\partial D)}+\|\hat{\bJ}_n\|_{L^2(B_R\backslash\overline{\Omega})^3}
\big)\\
&=C_8\big(\|\bnu\times\hat{\bE}^i_n\big|^+_{\partial \Omega}\|_{H^{-1/2}_{\mathrm{div}}(\partial \Omega)}+\|\bnu\times(\nabla\times\hat{\bE}^i_n)\big|^+_{\partial \Omega}\|_{H^{-1/2}_{\mathrm{div}}(\partial \Omega)}\\
&\qquad+\|\bnu\times(\delta\nabla\times\hat{\bE}_n)|^{-}_{\partial D}\|_{H^{-1/2}_{\mathrm{div}}(\partial D)}+\|\hat{\bJ}_n\|_{L^2(B_R\backslash\overline{\Omega})^3}
\big)\\
&\leq C_9\big(\|\hat{\bE}^i_n\|_{H^1(\mathrm{curl},B_R\backslash\overline{\Omega})}+\|\nabla\times\hat{\bE}^i_n\|_{H^1(\mathrm{curl},B_R\backslash\overline{\Omega})}+\|\delta\hat{\bE}_n\|_{H^1(\mathrm{curl},D)}\nonumber\\
&\qquad+\|\hat{\bJ}_n\|_{L^2(B_R\backslash\overline{\Omega})^3}\big)\\
&=C_9\sqrt{\delta}\sqrt{\delta}\|\hat{\bE}_n\|_{H^1(\mathrm{curl}, D)}\leq C_9\sqrt{\delta}\to 0\quad \mbox{as}\,\,\to 0,
 \end{align*}
 which contradicts with the fact that $\|\hat{\bE}_n\|_{H^1(\mathrm{curl},B_R\backslash\overline{D})}=1$.
Here, $C_8$ and $C_9$ are positive constants not relying on $\delta$. Hence, the inequality \eqref{inq:esti2_m} holds.

In the following, we prove \eqref{inq:esti1_m}. From \eqref{inq:esti2_m} and \eqref{eq:inter8 2} it is easy to see that \eqref{inq:esti1_m} holds. We finish the proof.
\end{proof}

\subsection{Results for Case 2 of Theorem \ref{thm:main1}}
In this subsection, we also address the medium scattering system (\ref{ME:tran1}), but with parameters $\mu_D$ and $\varepsilon_D$ chosen as in (\ref{eq:eff2_2}). Analogous to Lemma \ref{lem:mediu1}, we derive the following lemma.

\begin{lem}\label{lem:mediu2}
Let $\bE_\delta$ be the solution of \eqref{ME:tran1_E} for Case 2. Then there are positive constants $\widetilde{C}_1$ and $\widetilde{C}_2$ such that the following  estimates hold for all $\delta\ll 1$ and sufficiently large $R$:
\begin{align}
&\|\bE_\delta\|_{H^1(\mathrm{curl},D)}\leq \widetilde{C}_1\delta^{1/2}\left(\|\nabla\times\bE^i\|_{H^1(\mathrm{curl},B_R\backslash\overline{D})} +\|\bE^i\|_{H^1(\mathrm{curl},B_R\backslash\overline{\Omega})} + \|\bJ\|_{L^2(B_R\backslash\overline{D})^3} \right),\label{inq:esti1_m'}\\
&\|\bE_\delta\|_{H^1(\mathrm{curl},B_R\backslash\overline{D})}\leq \widetilde{C}_2\left(\|\nabla\times\bE^i\|_{H^1(\mathrm{curl},B_R\backslash\overline{D})} +\|\bE^i\|_{H^1(\mathrm{curl},B_R\backslash\overline{\Omega})} + \|\bJ\|_{L^2(B_R\backslash\overline{D})^3} \right).\label{inq:esti2_m'}
\end{align}
\end{lem}
\begin{proof} Similarly, we have
{\small\begin{align*}
\int_{B_R\backslash\overline{\Omega}}\mathrm{i}k\bJ\cdot\overline{\bE}^s_\delta\rmd\bx=&\mu_D^{-1}\,\int_{D}|\nabla\times\bE_\delta|^2\mathrm{d}\bx-k^2\varepsilon_D\int_{D}|\bE_\delta|^2\mathrm{d}\bx\,+\int_{\Omega\backslash\overline{D}}\mu^{-1}_r|\nabla\times\bE_\delta|^2\mathrm{d}\bx\nonumber\\
&-k^2\int_{\Omega\backslash\overline{D}}\varepsilon_r\bE_\delta\cdot\overline{\bE}_\delta\mathrm{d}\bx+\int_{B_R\backslash\overline{\Omega}}|\nabla\times\bE^s_\delta|^2\mathrm{d}\bx-\int_{B_R\backslash\overline{\Omega}}k^2|\bE^s_\delta|^2\mathrm{d}\bx
\nonumber\\
&+\int_{\partial \Omega}\gamma_t(\nabla\times\bE^s_\delta )\cdot(\overline{\bE}^i)_{\mathrm{T}} \rmd\sigma+\int_{\partial \Omega}\gamma_t(\nabla\times\bE^i )\cdot(\overline{\bE}^i)_{\mathrm{T}} \rmd\sigma\\
&+\int_{\partial \Omega}\gamma_t(\nabla\times\bE^i )\cdot(\overline{\bE}_\delta^s)_{\mathrm{T}} \rmd\sigma
+\int_{\partial B_R}\gamma_t(\nabla\times\bE^s_\delta )\cdot(\overline{\bE}_\delta^s)_{\mathrm{T}}  \rmd\sigma.
\end{align*}}
Noting that the expressions of $\mu_D$ and $\varepsilon_D$ are given by \eqref{eq:eff2_2}, and then
taking the real and imaginary parts of this integral identity, we can easily obtain
{\small\begin{align}
k\Im\int_{B_R\backslash\overline{\Omega}}\bJ\cdot\overline{\bE}^s_\delta\rmd\bx=&\delta^{-1}\,\|\nabla\times\bE_\delta\|^2_{L^2(D)^3}-k^2\eta_0\|\bE_\delta\|^2_{L^2(D)^3}+\int_{\Omega\backslash\overline{D}}\mu^{-1}_r|\nabla\times\bE_\delta|^2\mathrm{d}\bx\nonumber\\
&-k^2\int_{\Omega\backslash\overline{D}}\Re\varepsilon_r\bE_\delta\cdot\overline{\bE}_\delta\mathrm{d}\bx+\|\nabla\times\bE^s_\delta\|^2_{L^2(B_R\backslash\overline{\Omega})^3}-k^2\|\bE^s_\delta\|^2_{L^2(B_R\backslash\overline{\Omega})^3}
\nonumber\\
&+\Re\int_{\partial \Omega}\gamma_t(\nabla\times\bE^s_\delta )\cdot(\overline{\bE}^i)_{\mathrm{T}} \rmd\sigma+\Re\int_{\partial \Omega}\gamma_t(\nabla\times\bE^i )\cdot(\overline{\bE}^i)_{\mathrm{T}} \rmd\sigma\nonumber\\
&+\Re\int_{\partial \Omega}\gamma_t(\nabla\times\bE^i )\cdot(\overline{\bE}_\delta^s)_{\mathrm{T}}\rmd\sigma+\Re\int_{\partial B_R}\gamma_t(\nabla\times\bE^s_\delta )\cdot(\overline{\bE}_\delta^s)_{\mathrm{T}}  \rmd\sigma \label{eq:inter4-2}
\end{align}}
and
\begin{align}
k\Re\int_{B_R\backslash\overline{\Omega}}\bJ\cdot\overline{\bE}^s_\delta\rmd\bx=&\Im\int_{\partial \Omega}\gamma_t(\nabla\times\bE^s_\delta )\cdot(\overline{\bE}^i)_{\mathrm{T}} \rmd\sigma-\Im\int_{\partial \Omega}\gamma_t(\nabla\times\bE^i )\cdot(\overline{\bE}_\delta^s)_{\mathrm{T}}\rmd\sigma\nonumber\\
&+\Im\int_{\partial \Omega}\gamma_t(\nabla\times\bE^i )\cdot(\overline{\bE}^i)_{\mathrm{T}} \rmd\sigma-k^2\tau_0\delta^{-1}\|\bE_\delta\|^2_{L^2(D)^3}\nonumber\\
&-k^2\int_{\Omega\backslash\overline{D}}\Im\varepsilon_r\,\bE_\delta\cdot\overline{\bE}_\delta\mathrm{d}\bx+\Im\int_{\partial B_R}\gamma_t(\nabla\times\bE^s_\delta )\cdot(\overline{\bE}_\delta^s)_{\mathrm{T}}  \rmd\sigma. \label{eq:inter6 2-2}
\end{align}
Combining \eqref{eq:inter4-2}--\eqref{eq:inter6 2-2} with the continuities of the operators $\gamma_t$ and $\gamma_T$, the following inequalities can be directly derived,
\begin{align*}
\|\nabla\times\bE_\delta\|^2_{L^2(D)^3}\leq & \widetilde{C}_1\,\delta\,\Big(\|\bE_\delta\|^2_{L^2(D)^3}+\|\nabla\times\bE_\delta\|^2_{L^2(\Omega\backslash\overline{D})^3}
 +\|\bE_\delta\|_{L^2(\Omega\backslash\overline{D})^3}\nonumber\\
 &+\|\nabla\times\bE^s_\delta\|^2_{L^2(B_R\backslash\overline{\Omega})^3}+\|\nabla\times\bE^s_\delta\|_{H^1(\mathrm{curl},B_R\backslash\overline{\Omega})} \|\bE^i\|_{H^1(\mathrm{curl},B_R\backslash\overline{\Omega})} \nonumber\\
 &+\|\bE_\delta\|^2_{L^2(B_R\backslash\overline{\Omega})^3} +   \|\nabla\times\bE^i\|_{H^1(\mathrm{curl},B_R\backslash\overline{\Omega})} \|\bE^i\|_{H^1(\mathrm{curl},B_R\backslash\overline{\Omega})}\nonumber\\
 &+\|\nabla\times\bE^i\|_{H^1(\mathrm{curl},B_R\backslash\overline{\Omega})} \|\bE^s_\delta\|_{H^1(\mathrm{curl},B_R\backslash\overline{\Omega})}+\|\bJ\|_{L^2(B_R\backslash\overline{\Omega})^3}\|\bE^s_\delta\|_{L^2(B_R\backslash\overline{\Omega})^3}\nonumber\\
 &+\|\nabla\times\bE^s_\delta\|_{H^1(\mathrm{curl},B_R\backslash\overline{\Omega})} \|\bE^s_\delta\|_{H^1(\mathrm{curl},B_R\backslash\overline{\Omega})}
 \Big)\nonumber\\
\leq & \widetilde{C}_2\,\delta\,\Big(\|\bE_\delta\|^2_{L^2(D)^3}+\|\bE_\delta\|^2_{H^1(\mathrm{curl},B_R\backslash\overline{D})}+\|\nabla\times\bE^i\|^2_{H^1(\mathrm{curl},B_R\backslash\overline{\Omega})}\nonumber\\
&\quad+\|\bE^i\|^2_{H^1(\mathrm{curl},B_R\backslash\overline{\Omega})}+\|\bJ\|^2_{L^2(B_R\backslash\overline{D})^3}\Big)
\end{align*}
and
 \begin{align*}
\|\bE_\delta\|^2_{L^2(D)^3}\leq &\widetilde{C}_3\,\delta\,\Big( \|\bE_\delta\|^2_{L^2(\Omega\backslash\overline{D})^3}  +\|\nabla\times\bE^i\|_{H^1(\mathrm{curl},B_R\backslash\overline{\Omega})} \|\bE^i\|_{H^1(\mathrm{curl},B_R\backslash\overline{\Omega})}\nonumber\\
&+\|\nabla\times\bE^s_\delta\|_{H^1(\mathrm{curl},B_R\backslash\overline{\Omega})} \|\bE^i\|_{H^1(\mathrm{curl},B_R\backslash\overline{\Omega})}++\|\bJ\|_{L^2(B_R\backslash\overline{\Omega})^3}\|\bE^s_\delta\|_{L^2(B_R\backslash\overline{\Omega})^3}\nonumber\\
&+\|\nabla\times\bE^i\|_{H^1(\mathrm{curl},B_R\backslash\overline{\Omega})} \|\bE^s_\delta\|_{H^1(\mathrm{curl},B_R\backslash\overline{\Omega})}\nonumber\\
&+\|\nabla\times\bE^s_\delta\|_{H^1(\mathrm{curl},B_R\backslash\overline{\Omega})} \|\bE^s_\delta\|_{H^1(\mathrm{curl},B_R\backslash\overline{\Omega})} \big)\nonumber\\[5pt]
 \leq &\widetilde{C}_4\,\delta\,\big( \|\bE_\delta\|_{H^1(\mathrm{curl},B_R\backslash\overline{D})} +\|\bE^i\|_{H^1(\mathrm{curl},B_R\backslash\overline{\Omega})}+ \|\nabla\times\bE^i\|_{H^1(\mathrm{curl},B_R\backslash\overline{\Omega})} \nonumber\\
 &\quad+\|\bJ\|^2_{L^2(B_R\backslash\overline{\Omega})^3}\Big),
\end{align*}
where $\widetilde{C}_1$, $\widetilde{C}_2$, $\widetilde{C}_3$, and $\widetilde{C}_4$ are positive constants depending on $q,\,\eta_0,\tau_0, \,B_R,\, k,\,\Omega,\,\mu_r$ and $D$. We note that $(\delta^2+\delta)\leq 2\delta$ for enough small $\delta$, so
\begin{align}\label{eq:inter8 2-2}
\|\bE_\delta\|^2_{H^1(\mathrm{curl},D)} \leq &\widetilde{C}_5 \,\delta\,\Big( \|\bE_\delta\|^2_{H^1(\mathrm{curl},B_R\backslash\overline{D})}  +\|\bE^i\|^2_{H^1(\mathrm{curl},B_R\backslash\overline{\Omega})} +\|\nabla\times\bE^i\|^2_{H^1(\mathrm{curl},B_R\backslash\overline{\Omega})}\nonumber\\
&\quad+ \|\bJ\|^2_{L^2(B_R\backslash\overline{\Omega})^3}\Big),
\end{align}
where $\widetilde{C}_5$ does not reply on $\delta$. In what follows, the estimate \eqref{inq:esti2_m'} can be proved by contradiction. Assume that for any nonnegative integer $n$, there exists a set of data $\{(\bE_n, \bE^i_n,\bJ_n)\}$, where $\bE_n$ is the unique solution of \eqref{ME:tran1_E} with $\bJ_n$ and $\bE^i_n$ as inputs and satisfies the following conditions
\begin{align*}
\left\{ \begin{array}{ll}
\|\bJ_n\|_{L^2(B_R\backslash\overline{\Omega})^3}+\|\bE^i_n\|_{H^1(\mathrm{curl},B_R\backslash\overline{\Omega})^3}+\|\nabla\times\bE^i_n\|_{H^1(\mathrm{curl},B_R\backslash\overline{\Omega})^3}=1	,\quad \\[5pt]
\|\bE_n\|_{H^1(\mathrm{curl},B_R\backslash\overline{D})}\rightarrow\infty \quad \mbox{as}\quad \delta\to 0.
\end{array}
 \right.
\end{align*}
Another set of data $\{(\hat{\bE}_n, \hat{\bE}^i_n,\hat{\bJ}_n)\}$ is constructed as follows
\begin{align*}
\left\{ \begin{array}{ll}
\hat{\bJ}_n=\dfrac{\bJ_n}{\|\bE_n\|_{H^1(\mathrm{curl},B_R\backslash\overline{D})}},\quad \hat{\bE}_n=\dfrac{\bE_n}{\|\bE_n\|_{H^1(\mathrm{curl},B_R\backslash\overline{D})}},\\[15pt]
\hat{\bE}^i_n=\dfrac{\bE^i_n}{\|\bE_n\|_{H^1(\mathrm{curl},B_R\backslash\overline{D})}},\quad \hat{\bE}^s_n=\dfrac{\bE_n-\bE^i_n}{\|\bE_n\|_{H^1(\mathrm{curl},B_R\backslash\overline{D})}}.
\end{array}
 \right.
\end{align*}
It is easy to verify that $\hat{\bE}_n$ solves the system \eqref{ME:tran1_E} associated  with the input data $\hat{\bE}^i_n$ and $\hat{\bJ}_n$. As $\delta$ goes to $0$, we have
\begin{align*}
\|\hat{\bE}_n\|_{H^1(\mathrm{curl},B_R\backslash\overline{D})}=1 \,\,\mbox{and}\quad
 \|\hat{\bJ}_n\|_{L^2(B_R\backslash\overline{D})^3},\,\|\nabla\times\hat{\bE}^i_n\|_{H^1(\mathrm{curl},B_R\backslash\overline{D})},\, \|\hat{\bE}^i_n\|_{H^1(\mathrm{curl},B_R\backslash\overline{D})}\to 0 .
\end{align*}
Similarly, we obtain the following estimate,
\begin{align*}
\|\hat{\bE}_n\|_{H^1(\mathrm{curl},D)} \leq &\widetilde{C}_6 \sqrt{\delta}\,\big( \|\hat{\bE}_n\|^2_{H^1(\mathrm{curl},B_R\backslash\overline{D})}  +\|\hat{\bE}^i_n\|^2_{H^1(\mathrm{curl},B_R\backslash\overline{\Omega})} +\|\nabla\times\hat{\bE}^i_n\|^2_{H^1(\mathrm{curl},B_R\backslash\overline{\Omega})}\nonumber \\
&\qquad\quad+\|\hat{\bJ}_n\|^2_{L^2(B_R\backslash\overline{\Omega})^3}\big)^{1/2}
\end{align*}
where $\widetilde{C}_6$ is a positive constant not relying on $\delta$.

Let $\hat{\bH}_n=\frac{\hat{\bE}_n}{\mathrm{i}k\mu_r}$ and $\hat{\bH}^s_n=\frac{\hat{\bE}^s_n}{\mathrm{i}k}$. Indeed, $(\hat{\bE}_n\big|_{\Omega\backslash\overline{D}},\,\hat{\bH}_n\big|_{\Omega\backslash\overline{D}},\,\hat{\bE}^s_n\big|_{\mathbb{R}^3\backslash\overline{\Omega}},\,\hat{\bH}^s_n\big|_{\mathbb{R}^3\backslash\overline{\Omega}})$ is the unique solution of \eqref{MD:tran2} with these boundary conditions $\bp=\bnu\times \hat{\bE}_n\Big|^+_{\partial{D}}$, $\bff=\bnu\times\hat{\bE}^i_n\big|^+_{\partial \Omega}$ and $\mathrm{i}k \bh=\bnu\times(\nabla\times\hat{\bE}^i_n)\big|^+_{\partial \Omega}$. Combining with Theorem \ref{th:exitence2}, one implies
 \begin{align*}
\|\hat{\bE}_n\|_{H^1(\mathrm{curl},B_R\backslash\overline{D})}\leq & \widetilde{C}_7\big(\|\bnu\times\hat{\bE}^i_n\big|^+_{\partial \Omega}\|_{\mathrm{div}(\partial \Omega)}+\big\|\frac{1}{\mathrm{i}k}\bnu\times(\nabla\times\hat{\bE}^i_n)\big|^+_{\partial \Omega}\big\|_{H^{-1/2}_{\mathrm{div}}(\partial \Omega)}\\
&\qquad+\|\bnu\times(\mu_r^{-1}\nabla\times \hat{\bE}_n)\big|^+_{\partial{D}}\|_{H^{-1/2}_{\mathrm{div}}(\partial D)}+\|\hat{\bJ}_n\|_{L^2(B_R\backslash\overline{\Omega})^3}
\big)\\
\leq& \widetilde{C}_8\big(\|\hat{\bE}_n\|_{H^1(\mathrm{curl}, D)}+\|\hat{\bE}^i_n\|_{H^1(\mathrm{curl},B_R\backslash\overline{\Omega})}+\|\nabla\times\hat{\bE}^i_n\|_{H^1(\mathrm{curl},B_R\backslash\overline{\Omega})}\nonumber\\
&\quad+\|\hat{\bJ}_n\|_{L^2(B_R\backslash\overline{\Omega})^3}
\big)
 \end{align*}
for some positive constants $\widetilde{C}_7$ and $\widetilde{C}_8$.
As $\delta$ goes to $0$, we see that $\|\hat{\bE}_n\|_{H^1(\mathrm{curl},B_R\backslash\overline{D})}\to 0$. This contradicts with  the fact that $\|\hat{\bE}_n\|_{H^1(\mathrm{curl},B_R\backslash\overline{D})}=1$. At this point, we have proven \eqref{inq:esti2_m'}. By combining \eqref{inq:esti2_m'} and \eqref{eq:inter8 2-2}, it is straightforward to prove \eqref{inq:esti1_m'}. We complete the proof.
\end{proof}

\section{Proofs of Theorem \ref{thm:main1} and Theorem \ref{thm:main2}}\label{sect:proof}
Before delving into the proof details of Theorem \ref{thm:main1}, we establish preliminary estimates concerning the electric field in the medium scattering problem (\ref{ME:tran1}) for both Case 1 and Case 2. However, for Case 1, we require the following proposition, which directly follows from Lemma \ref{lem:mediu1}.

\begin{prop}\label{prop:der_case1}
Assuming that $\bE_\delta\in H^1_{\rm loc}(\mathrm{curl}, \mathbb{R}^3)$ is the solution of \eqref{ME:tran1_E} for Case 1, then there is $\delta_0>0$ satisfying the following estimate for $\delta<\delta_0$:
\begin{align*}
\left\|\bnu\times(\mu_r^{-1}\nabla\times \bE_\delta)\big|^+_{\partial{D}}\right\|_{H^{-1/2}_{\mathrm{div}}(\partial D)}\leq & C\sqrt{\delta}\big(\|\bE^i\|_{H^1(\mathrm{curl}, \mathbb{R}^3\backslash\overline{\Omega})}+\|\nabla\times\bE^i\|_{H^1(\mathrm{curl}, \mathbb{R}^3\backslash\overline{\Omega})}\nonumber\\
&\quad+\|\bJ\|_{L^2(B_R\backslash\overline{\Omega})^3}\big).
\end{align*}
\end{prop}

Subsequently, we demonstrate that the solution of \eqref{eq:sca1} can be approximated by the solution of \eqref{ME:tran1} with respect to the parameter $\delta$.

\begin{prop}\label{diff}
Let $(\bE,\bH)\in H^1_{\rm loc}(\mathrm{curl},\mathbb{R}^3\backslash\overline{D})\times H^1_{\rm loc}(\mathrm{curl},\mathbb{R}^3\backslash\overline{D})$ be the solution of \eqref{eq:sca1} and $(\bE_\delta,\bH_\delta)\in H^1_{\rm loc}(\mathrm{curl},\mathbb{R}^3)\times H^1_{\rm loc}(\mathrm{curl},\mathbb{R}^3)$ be the solution of \eqref{ME:tran1} with the incident wave $\bE^i$ and source $\bJ$. Then,  for a sufficiently large $R$ and sufficiently small $\delta$, we have
\begin{align*}
\|\bE-\bE_\delta\|_{H^1(\mathrm{curl},B_R\backslash\overline{D})}+\|\bH-\bH_\delta\|_{H^1(\mathrm{curl},B_R\backslash\overline{D})}\leq & C\sqrt{\delta}\big(\|\bE^i\|_{H^1(\mathrm{curl}, B_R\backslash\overline{\Omega})}+\|\bJ\|_{L^2(B_R\backslash\overline{\Omega})^3}\nonumber\\
&\quad+\|\nabla\times\bE^i\|_{H^1(\mathrm{curl}, B_R\backslash\overline{\Omega})}\big).
\end{align*}
\end{prop}
\begin{proof} Certainly, we will divide this proof into two parts corresponding to different boundary conditions for $D$.

\medskip
\noindent {\bf Part I.}
Consider the boundary condition on $\partial D$ of \eqref{eq:sca1}(i.e., $\mathfrak{B}(\bE)=\bnu\times\big(\mu^{-1}_r\nabla\times \bE\big)=\bf{0}$) and the material parameters in the region $D$ associated with the system \eqref{ME:tran1} described as in Case 1 of Theorem \ref{thm:main1}.  Denote $\widetilde{\bE}=\bE_\delta-\bE$, $\widetilde{\bH}=\bH_\delta-\bH$, $\widetilde{\bE}^s=\bE_\delta^s-\bE^s$ and $\widetilde{\bH}^s=\bH_\delta^s-\bH^s$. It is easy to verify that $(\widetilde{\bE},\widetilde{\bH},\widetilde{\bE}^s,\widetilde{\bH}^s)\in H^1(\mathrm{curl}, \Omega\backslash\overline{D})\times H^1(\mathrm{curl}, \Omega\backslash\overline{D})\times  H^1_{\mathrm{loc}}(\mathrm{curl}, \mathbb{R}^3\backslash\overline{\Omega})\times H^1_{\mathrm{loc}}(\mathrm{curl}, \mathbb{R}^3\backslash\overline{\Omega})$ is the unique solution of \eqref{MD:tran2} with the boundary conditions: $\bff=\bh=\bJ=0$ and $\bq=\bnu\times(\mu_r^{-1}\nabla\times\widetilde{\bE})$.
According to  Proposition \ref{prop:der_case1}, Theorem \ref{th:exitence1} and  Lemma \ref{lem:mediu1}, we get
\begin{align*}
&\|\widetilde\bE\|_{H^1(\mathrm{curl},\Omega\backslash\overline{D})}+\|\widetilde\bH\|_{H^1(\mathrm{curl},\Omega\backslash\overline{D})}+\|\widetilde\bE^s\|_{H^1(\mathrm{curl}, B_R\backslash\overline{\Omega})}+\|\widetilde\bH^s\|_{H^1(\mathrm{curl}, B_R\backslash\overline{\Omega})}\nonumber\\
&\qquad\quad\leq C\|\bq\|_{H^{-1/2}_{\mathrm{div}}(\partial D)}= C \|\bnu\times(\mu_r^{-1}\nabla\times\bE_\delta)\big|^-_{\partial D}\|_{H^{-1/2}_{\mathrm{div}}(\partial D)}\\
&\qquad\quad\leq C\sqrt{\delta}\big(\|\bE^i\|_{H^1(\mathrm{curl}, \mathbb{R}^3\backslash\overline{\Omega})}+\|\nabla\times\bE^i\|_{H^1(\mathrm{curl}, \mathbb{R}^3\backslash\overline{\Omega})}+\|\bJ\|_{L^2(B_R\backslash\overline{\Omega})^3}\big)
\end{align*}
for some positive constants $C$  that do not depend on $\delta$.

\medskip
\noindent {\bf Part II.}
Consider the boundary condition on $\partial D$ of \eqref{eq:sca1}(i.e., $\mathfrak{B}(\bE)=\bnu\times \bE=\bf{0}$) and the material parameters in the region $D$ associated with the system \eqref{ME:tran1} described as in Case 2 of Theorem \ref{thm:main1}. Denote $\widetilde{\bE}=\bE_\delta-\bE$, $\widetilde{\bH}=\bH_\delta-\bH$, $\widetilde{\bE}^s=\bE_\delta^s-\bE^s$ and $\widetilde{\bH}^s=\bH_\delta^s-\bH^s$. We can easily obtain that $(\widetilde{\bE},\widetilde{\bH},\widetilde{\bE}^s,\widetilde{\bH}^s)\in H^1(\mathrm{curl}, \Omega\backslash\overline{D})\times H^1(\mathrm{curl}, \Omega\backslash\overline{D})\times  H^1_{\mathrm{loc}}(\mathrm{curl}, \mathbb{R}^3\backslash\overline{\Omega})\times H^1_{\mathrm{loc}}(\mathrm{curl}, \mathbb{R}^3\backslash\overline{\Omega})$ is the unique solution of \eqref{MD:tran2-2}, where $\bff=\bh=\bJ=0$ and $\bq=\bnu\times\widetilde{\bE}\big|_{\partial D}=\bnu\times\bE_\delta\big|_{\partial D}$. By using Lemma \ref{lem:mediu2}, trace Theory and Theorem \ref{th:exitence2}, it is straightforward to imply that
\begin{align*}
&\|\widetilde\bE\|_{H^1(\mathrm{curl},\Omega\backslash\overline{D})}+\|\widetilde\bH\|_{H^1(\mathrm{curl},\Omega\backslash\overline{D})}+\|\widetilde\bE^s\|_{H^1(\mathrm{curl}, B_R\backslash\overline{\Omega})}+\|\widetilde\bH^s\|_{H^1(\mathrm{curl}, B_R\backslash\overline{\Omega})}\nonumber\\
&\qquad\quad\leq C'\|\bq\|_{H^{-1/2}_{\mathrm{div}}(\partial D)}= C' \|\bnu\times\bE_\delta\big|^-_{\partial D}\|_{H^{-1/2}_{\mathrm{div}}(\partial D)}\leq C''\|\bE_\delta\|_{H^1(\mathrm{curl}, D)}\\
&\qquad\quad\leq C'''\sqrt{\delta}(\|\bE^i\|_{H^1(\mathrm{curl}, \mathbb{R}^3\backslash\overline{\Omega})}+\|\nabla\times\bE^i\|_{H^1(\mathrm{curl}, \mathbb{R}^3\backslash\overline{\Omega})}+\|\bJ\|_{L^2(B_R\backslash\overline{\Omega})^3}),
\end{align*}
where $C',C''$ and $C'''$ are positive constants without depending on $\delta$. So we finish the proof.
\end{proof}

After completing the necessary preparations outlined above, we will now proceed to prove Theorem \ref{thm:main1}.

\begin{proof}[Proof of Theorem \ref{thm:main1}]
 Let $\widetilde{\bE}=\bE_\delta-\bE$ and $\widetilde{\bE}^s=\bE_\delta^s-\bE^s$. Note that $\widetilde{\bE}^s=\widetilde{\bE}$. We adopt these explicit expressions for the scattering amplitude of $\bE_\delta^s$ and $\bE^s$ (cf.\cite{{Colton2013}}):
\begin{eqnarray*}
\bE_\delta^{\infty}(\hat{\bx})=\frac{\mathrm{i}k}{4\pi}\int_{\partial \Omega}\Big\{\bnu(\by)\times\bE_\delta^s(\by)+\Big(\bnu(\by)\times\frac{\nabla\times\bE^s_\delta(\by)}{\mathrm{i}k}\Big)\times\hat{\bx}\Big\}e^{-\mathrm{i}k\hat{\bx}\cdot\by}\rmd \sigma(y),\quad\hat{\bx}\in \mathbb{S}^{2}, \\
\bE^{\infty}(\hat{\bx})=\frac{\mathrm{i}k}{4\pi}\int_{\partial \Omega}\Big\{\bnu(\by)\times\bE^s(\by)+\Big(\bnu(\by)\times\frac{\nabla\times\bE^s(\by)}{\mathrm{i}k}\Big)\times\hat{\bx}\Big\}e^{-\mathrm{i}k\hat{\bx}\cdot\by}\rmd \sigma(y),\quad\hat{\bx}\in \mathbb{S}^{2}.
\end{eqnarray*}
Then $\widetilde{\bE}^{\infty}(\hat{\bx})=\frac{\mathrm{i}k}{4\pi}\int_{\partial \Omega}\big\{\bnu(\by)\times\widetilde{\bE}^s(\by)+\big(\bnu(\by)\times\frac{\nabla\times\widetilde{\bE}^s(\by)}{\mathrm{i}k}\big)\times\hat{\bx}\big\}e^{-\mathrm{i}k\hat{\bx}\cdot\by}\rmd \sigma(y).$ Thus, combined with Proposition \ref{diff}, we can derive the following estimate
\begin{align*}
\big\|\widetilde{\bE}^{\infty}(\hat{\bx})\big\|_{L(\mathbb{S}^2)^3}\leq&
\Big\|\frac{\mathrm{i}k}{4\pi}\hat{\bx}\times\int_{\partial\Omega}\Big(\bnu(\by)\times\frac{\nabla\times\widetilde{\bE}^s(\by)}{\mathrm{i}k}\Big)\times\hat{\bx}\,\,e^{-\mathrm{i}k\hat{\bx}\cdot\by}\rmd \sigma(y)\Big\|_{L(\mathbb{S}^2)^3}\\
&\quad+\Big\| \frac{\mathrm{i}k}{4\pi}\hat{\bx}\times\int_{\partial\Omega}\bnu(\by)\times\bE^s(\by)\,\,e^{-\mathrm{i}k\hat{\bx}\cdot\by} \mathrm{d}\sigma(\by) \Big\|_{L(\mathbb{S}^2)^3}\\
\leq& M_1\|\widetilde{\bE}^s\|_{L^2(B_R\backslash\overline{\Omega})^3}+M_2\|\nabla\times\widetilde{\bE}^s\|_{L^2(B_R\backslash\overline{\Omega})^3}\\
\leq& M_3\|\widetilde{\bE}^s\|_{H^1(\mathrm{curl},B_R\backslash\overline{\Omega})}\leq M_4 \sqrt{\delta}\Big(\|\bE^i\|_{H^1(\mathrm{curl}, \mathbb{R}^3\backslash\overline{\Omega})}\nonumber\\
&\quad+\|\nabla\times\bE^i\|_{H^1(\mathrm{curl}, \mathbb{R}^3\backslash\overline{\Omega})}+\|\bJ\|_{L^2(B_R\backslash\overline{\Omega})^3}\Big),
\end{align*}
where $M_1,\,M_1,\, M_3$ and $M_4$ are positive constants depending only on $\omega$, $k$, $\gamma$, $B_R\backslash\overline{\Omega}$ and $B_R\backslash\overline{D}$. We finish the proof.
\end{proof}
\begin{proof}[Proof of Theorem \ref{thm:main2}]Using analogous arguments to those presented in the proofs of Theorem \ref{thm:main1} for Case 1 and Case 2, we make the necessary adjustments to accommodate our current scenario. To avoid redundancy, we provide a summary of the proof process specifically tailored for the case where $D$ represents the complex PMC obstacles.

\medskip
 \noindent {\bf Step 1:} Demonstrate the uniqueness of the solution for a modified scattering problem \eqref{MD:tran2-2}, where the boundary conditions on $\partial D$ are revised as follows:
{\small\begin{align*}
\bnu\times \bH\big|_{ \bigcup_{l=1}^{N}\partial D_l\cap\partial D} =\bq\in H^{-1/2}_{\mathrm{div}}\big(\cup_{l=1}^{N}\partial D_{k}\cap\partial D\big).
\end{align*}}
The uniqueness of this modified problem is established using a similar approach as in Theorem \ref{th:unique}.

\medskip
 \noindent {\bf Step 2:} Prove the existence of a solution for the modified scattering problem \eqref{MD:tran2-2} and establish an estimate relating the solution to the boundary data and $\bJ$.
The process is similar to that described in the proof of Theorem \ref{th:exitence1}.

\medskip
\noindent {\bf Step 3:} Derive the related estimates of the the effective medium scattering problems \eqref{ME:tran1} with $D=\cup_{l=1}^N D_l$. By multiplying $\nabla\times(\widetilde{\mu}_r\,\nabla\times \,\bE_\delta)-k^2\widetilde{\varepsilon}_r\, \bE_\delta=\mathrm{i}k\bJ$ by $\overline{\bE}_\delta$ in $D$, we obtain the equation
$$\bigcup_{l=1}^N\left\{\int_{D_l}\mu^{-1}_{D_l}|\nabla\times\bE_\delta|^2\mathrm{d}\bx+\int_{\partial D_l}\delta\bnu\times( \nabla\times \bE_\delta)\cdot\overline{\bE}_\delta\mathrm{d}\sigma-\int_{{D_l}}k^2\varepsilon_{D_l}|\bE_\delta|^2\mathrm{d}\bx\right\}=0.
$$
It is important to emphasize the following transmission conditions on the relevant boundaries:
\begin{itemize}
\item On $\partial D_l\cap\partial D$, the conditions are $\bnu\times\bE_\delta\big|^-_{\partial D_l\cap \partial D}=\bnu\times\bE_\delta\big|^+_{\partial D_l\cap \partial D}$ and $\bnu\times(\mu^{-1}_{D_l}\nabla\times\bE_\delta)\big|^-_{\partial D_l\cap \partial D}=\bnu\times(\mu_r^{-1}\nabla\times\bE_\delta)\big|^+_{\partial D_l\cap \partial D}$.
\item On $\partial D_l\backslash\overline{\partial D}$,  the conditions are $\bnu\times\bE_\delta\big|_{\partial D_l\cap D_j}=\bnu\times\bE_\delta\big|_{\partial D_l\cap\partial D_j}$ and $\bnu\times(\mu^{-1}_{D_l}\nabla\times\bE_\delta)\big|_{\partial D_l\cap \partial D_j}=\bnu\times(\mu^{-1}_{D_j}\nabla\times\bE_\delta)\big|_{\partial D_l\cap \partial D_j}$, where $\overline{D}_l\cap \overline{D}_j\neq \emptyset$ and $l\neq j$, $j=1,2,\cdots,N$.
\end{itemize}
By combining the integrals in $\Omega \backslash D$ and $B_R\backslash\overline{\Omega}$ with the fact that
$$\|\bE_\delta\|^2_{H^1(\mathrm{curl}, D)}=\sum^N_{l=1}\|\bE_\delta\|^2_{H^1(\mathrm{curl}, D_l)},$$
we can derive an inequality similar to \eqref{eq:inter8 2}.
The remaining process is similar to that described in the proof of Lemma \ref{lem:mediu1}.

\medskip
\noindent {\bf Step 4:} Obtain estimates similar to those in Proposition \ref{diff} by following a process similar to the one described in Proposition \ref{diff}.

\medskip
\noindent {\bf Step 5:} Prove  the difference between the two far fields can be controlled by the incident field $\bE^i$ and $\bJ$, indicating that its coefficient is the 1/2 power of $\delta$.
\end{proof}

\section*{Acknowledgements}

The work of H. Diao is supported by National Natural Science Foundation of China  (No. 12371422) and the Fundamental Research Funds for the Central Universities, JLU (No. 93Z172023Z01). The work of H. Liu is supported by the Hong Kong RGC General Research Funds (projects 12302919, 12301420 and 11300821), the NSFC/RGC Joint Research Fund (project  N\_CityU101/21), the France-Hong Kong ANR/RGC Joint Research Grant, A-HKBU203/19. The work of Q. Meng is supported by the Hong Kong RGC Postdoctoral Fellowship Scheme (No. 9061028).


%

\end{document}